\input ./style/arxiv-general.cfg
\documentclass[aos,MSNbibl,seceqn,dvips]{arximspdf}
\makeatletter
   \@ifpackageloaded{graphicx}{}{\usepackage{graphicx}}
\makeatother
\usepackage{mathbh}

\doi{10.1214/15-AOS1336}
\volume{43}
\issue{6}
\pubyear{2015}
\firstpage{2353}
\lastpage{2383}
\docsubty{FLA}

\makeatletter
\def\cal{\mathcal}

\newcommand{\eqref}[1]{(\ref{#1})}
\newcommand{\N}{\mathbb{N}}

\newcommand{\ep}{\epsilon}
\newcommand{\rn}{\sqrt{n}}
\newcommand{\veps}{\varepsilon}
\newcommand{\R}{\mathbb{R}}
\newcommand{\ga}{\gamma}
\def\1{\mathbh{1}}

\newcommand{\al}{\alpha}
\newcommand{\ta}{\tau}
\newcommand{\vphi}{\varphi}

\newcommand{\RR}{\mathbb{R}}

\newcommand{\cB}{{\mathcal{B}}}
\newcommand{\cC}{{\mathcal{C}}}
\newcommand{\cD}{{\mathcal{D}}}
\newcommand{\cF}{{\mathcal{F}}}
\newcommand{\cH}{{\mathcal{H}}}
\newcommand{\cK}{{\mathcal{K}}}
\newcommand{\cM}{{\mathcal{M}}}
\newcommand{\cN}{{\mathcal{N}}}
\newcommand{\cS}{{\mathcal{S}}}
\newcommand{\mb}{\mathbb{B}}
\newcommand{\mh}{\mathbb{H}}

\newcommand{\leqa}{\lesssim}
\newcommand{\geqa}{\gtrsim}

\newtheorem{thmm}{Theorem}[section]
\newproclaim{defi}[thmm]{Definition}
\newtheorem{prop}{Proposition}
\newtheorem{cor}{Corollary}

\newproclaim{rem}{Remark}
\newproclaim{example}{Example}[section]
\newproclaim{assumption}{Assumption}
\makeatother

\begin{document}
\begin{frontmatter}

\title{A Bernstein--von Mises theorem for smooth functionals in semiparametric models\thanksref{T1}}
\runtitle{A BvM theorem for smooth functionals}
\thankstext{T1}{Supported in part by ANR Grant projects ``Banhdits''
ANR-2010-BLAN-0113-03
and ``Calibration'' ANR 2011 BS01 010 01.}

\begin{aug}
\author[A]{\fnms{Isma\"el} \snm{Castillo}\corref{}\thanksref{m1,m2}\ead[label=e1]{ismael.castillo@math.cnrs.fr}}
\and
\author[B]{\fnms{Judith} \snm{Rousseau}\thanksref{m3,m4}\ead[label=e2]{rousseau@ceremade.dauphine.fr}}
\runauthor{I. Castillo and J. Rousseau}
\affiliation{CNRS--LPMA\thanksmark{m1},
Universit\'es Paris VI \& VII\thanksmark{m2},
CREST--ENSAE\thanksmark{m3}\\ and Universit\'e Paris Dauphine\thanksmark{m4}}
\address[A]{CNRS--LPMA \\
Universit\'es Paris VI \& VII \\
B\^{a}timent Sophie Germain \\
75205 Paris Cedex 13\\
France\\
\printead{e1}}

\address[B]{CREST--ENSAE\\
Laboratoire de Statistiques\\
3 avenue Pierre Larousse\\
92245 Malakoff Cedex\\
France\\
and\\
Universit\'e Paris Dauphine\\
Place du Mar\'echal DeLattre de Tassigny\\
75016 Paris\\
France\\
\printead{e2}}
\end{aug}

%
\received{\smonth{5} \syear{2013}}
%
\revised{\smonth{4} \syear{2015}}

%
\begin{abstract}
A Bernstein--von Mises theorem is derived for general semiparametric
functionals. The result is applied to a variety of semiparametric
problems in i.i.d. and non-i.i.d. situations. In particular, new tools
are developed to handle semiparametric bias, in particular for
nonlinear functionals and in cases where regularity is possibly low.
Examples include the squared $L^2$-norm in Gaussian white noise,
nonlinear functionals in density estimation, as well as functionals in
autoregressive models. For density estimation, a systematic study of
BvM results for two important classes of priors is provided, namely
random histograms and Gaussian process priors.
\end{abstract}

%

%
\end{frontmatter}

\section{Introduction}\label{sec1}

Bayesian approaches are often considered to be close asymptotically to
frequentist likelihood-based approaches so that the impact of the prior
disappears as the information brought by the data---typically the
number of observations---increases. This common knowledge is verified
in most parametric models, with a precise expression of it through the
so-called Bernstein--von Mises theorem or property (hereafter, BvM).
This property says that, as the number of observations increases the
posterior distribution can be approached by a Gaussian distribution
centered at an efficient estimator of the parameter of interest
and with variance the inverse of the Fisher information matrix of the
whole sample; see, for instance, van der Vaart \cite{aad98}, Berger
\cite{berger} or
Ghosh and Ramamoorthi \cite{ghosh:ramamo:2003}. The situation becomes,
however, more
complicated in non- and semiparametric models.
Semiparametric versions of the BvM property consider the behaviour of
the marginal posterior in a parameter of interest, in models
potentially containing an infinite-dimensional nuisance parameter.
There some care is typically needed in the choice of the nonparametric
prior and a variety of questions linked to prior choice and techniques
of proofs arise.
Results on semiparametric BvM applicable to general models and/or
general priors include Shen \cite{shen:2002}, Castillo \cite{icbvm},
Rivoirard and Rousseau \cite
{rivoirard:rousseau:09} and Bickel and Kleijn \cite{bickel:kleijn}.
The variety of
possible interactions between prior and model and the subtleties of
prior choice are illustrated in the previous general papers and in
recent results in specific models such as Kim \cite{kim06}, De Blasi
and Hjort \cite
{deblasi09}, Leahu \cite{leahu:11}, Knapik et al. \cite{kvv11},
Castillo \cite{icsan} and Kruijer and Rousseau \cite
{kruijer:rousseau:12}. In between semi- and nonparametric results, BvM
for parameters with growing dimension have been obtained in, for
example, Ghosal \cite{ghosal00}, Boucheron and Gassiat \cite
{boucheron:gassiat:10} and Bontemps \cite
{bontemps:11a}.
Finally, although there is no immediate analogue of the
BvM property for infinite dimensional parameters, as pointed out by
Cox \cite{C93} and Freedman \cite{Free99}, some recent contributions have
introduced possible notions of nonparametric BvM; see Castillo and
Nickl \cite{CN12} and
also Leahu \cite{leahu:11}. In fact, the results of the present paper are
relevant for these, as discussed below.

For semiparametric BvM, it is of particular interest to obtain generic
sufficient conditions that do not depend on the specific form of the
considered model.
In this paper, we provide a general result, Theorem~\ref{lem:laplace}
in Section~\ref{sec:main}, on the existence of the BvM property for
generic models and functionals of the parameter. Let us briefly discuss
the scope of our results; see Section~\ref{sec:main} for precise
definitions. Consider a model parameterised by $\eta$ varying in a
(subset of a) metric space $S$ equipped with a $\sigma$-field
$\mathcal
S$. Let $\psi: S \rightarrow\R^d$, $d \geq1$, be a measurable
functional of interest and let $\Pi$ be a probability distribution on $S$.
Given observations $Y^n$ from the model, we study the asymptotic
posterior distribution of $\psi(\eta)$, denoted $\Pi[\psi(\eta
)|
Y^n]$. Let $\cN(0,V)$ denote the centered normal law with covariance
matrix $V$. We give general conditions under which a BvM-type property
is valid,
%
\begin{equation}
\label{bvm-informal} \Pi \bigl[ \sqrt{n} \bigl( \psi(\eta) - \hat {\psi} \bigr) |
Y^n \bigr] \rightsquigarrow\cN(0,V),
\end{equation}
as $n\to\infty$ in probability,
where $\hat\psi$ is a (random) centering point, and $V$ a covariance
matrix, both to be specified, and where $\rightsquigarrow$ stands for
weak convergence. An interesting and well-known consequence of BvM is
that posterior credible sets, such as equal-tail credible intervals,
highest posterior density regions or one-sided credible intervals are
also confidence regions with the same asymptotic coverage.

The contributions of the present paper can be regrouped around the
following aims:
\begin{longlist}[1.]
\item[1.] Provide general conditions on the model and on the functional
$\psi$ to guarantee \eqref{bvm-informal} to hold, in a variety of
frameworks both i.i.d. and non-i.i.d. This includes investigating how
the choice of the prior influences bias $\hat\psi$ and
variance $V$. This also includes studying the case of nonlinear
functionals, which involves specific techniques for the bias. This is
done via a Taylor-type expansion of the functional involving a linear
term as well as, possibly, an additional quadratic term.
\item[2.] In frameworks with low regularity, second-order properties in the
functional expansion may become relevant. We study this as an
application of the main theorem in the important case of estimation of
the squared $L^2$-norm of an unknown regression function in the case
where the convergence rate for the functional is still parametric but
where the ``plug-in'' property in the sense of Bickel and Ritov \cite
{BR03} is not
necessarily satisfied.
\item[3.] Provide simple and ready-to-use sufficient conditions for BvM in
the important example of density estimation on the unit interval. We
present extensions and refinements
in particular of results of Castillo \cite{icbvm} and Rivoirard
and Rousseau \cite{rivoirard:rousseau:09} regarding, respectively, the use of Gaussian process priors in the context of density estimation,
and, the possibility to consider nonlinear functionals.
The class of random density histogram
priors is also studied in details systematically for the first time in
the context of Bayesian semiparametrics.
\item[4.] Provide simple sufficient conditions on the prior for BvM to hold
in a more complex example involving dependent data, namely the
nonlinear autoregressive model. To our knowledge, this is the first
result of this type in such a model.
\end{longlist}

The techniques and results of the paper, as it turned out, have also
been useful for different purposes in a recent series of works
developing a multiscale approach for posteriors, in particular:
(a) to prove functional limiting results, such as Bayesian versions of
Donsker's theorem, or more generally BvM results as in
Castillo and Nickl \cite{cn14}, a first step consists in proving the
result for finite
dimensional projections: this is exactly asking for a semiparametric
BvM to hold, and results from Section~\ref{sec:density} can be
directly applied;
(b) related to this is the study of many functionals simultaneously:
this is used in the study of posterior contraction rates in the
supremum norm in Castillo \cite{ic13}.
Finally, along the way, we shall also derive posterior rate results for
Gaussian processes which are of independent interest; see Proposition~2
in the supplemental article (Castillo and
Rousseau \cite{castillo:rousseau:suppl}).

Our results show that the most important condition is a \textit{no-bias}
condition, which will be seen to be essentially necessary. This
condition is written in a nonexplicit way in the general Theorem~\ref
{lem:laplace}, since the study of such a condition depends heavily on
the family of priors that are considered together with the statistical
model. Extensive discussions on the implication of this no-bias
condition are provided in the context of the white noise model and
density models for two families of priors. In the examples, we have
considered the main tool used to verify this condition consists in
constructing a change of parameterisation in the form $\eta\rightarrow
\eta+ \Gamma/\sqrt{n}$ for some given $\Gamma$ depending on the
functional of interest, which leaves the prior approximately unchanged.
Roughly speaking, for the no-bias condition to be valid, it is
necessary that both $\eta_0$ and $\Gamma$ are well approximated under
the prior. If this condition is not verified, then BvM may not hold: an
example of this phenomenon is provided in Section~\ref{counter}.

Theorem~\ref{lem:laplace} does not rely on a specific type of model,
nor on a specific family of functionals. In Section~\ref{sec:WNmodels},
it is applied to the study of a nonlinear functional in the white
noise model, namely the squared-norm of the signal. Applications to
density estimation with three different types of functionals and to an
autoregressive model can be found respectively in Section~\ref
{sec:density} and Section~\ref{sec:autoreg}. Section~\ref{sec:proofs}
is devoted to proofs, together with the supplemental article (Castillo
and Rousseau \cite
{castillo:rousseau:suppl}).

\subsection*{Model, prior and notation} \label{subsec:notation}

Let $(\mathcal Y^n, \mathcal G^n, P^n_\eta, \eta\in S)$ be a
statistical experiment, with observations $Y^n$ sitting on a space
$\mathcal Y^n$ equipped with a $\sigma$-field $\mathcal G^n$, and where
$n$ is an integer quantifying the available amount of information. We
typically consider the asymptotic framework $n\to\infty$. We assume
that $S$ is equipped with a $\sigma$-field $\mathcal S$, that $S$ is a
subset of a linear space and that for all $\eta\in S$, the measures
$P^n_\eta$ are absolutely continuous with respect to a dominating
measure $\mu_n$. Denote by $p_\eta^n$ the associated density and by
$\ell_n(\eta)$ the log-likelihood. Let $\eta_0$ denote the true value
of the parameter and $P^n_{\eta_0}$ the frequentist distribution of the
observations $Y^n$ under $\eta_0$. Throughout the paper, we set
$P_0^n:= P_{\eta_0}^n$ and $P_0 := P_0^1$. Similarly, $E_0^n[\cdot]$
and $E_0[\cdot]$ denote the expectation under $P_0^n$ and $P_0$,
respectively, and $E_\eta^n$ and $E_\eta$ are the corresponding
expectations under $P_\eta^n$ and $P_\eta$.
Given any prior probability $\Pi$ on $S$, we denote by $\Pi[\cdot|Y^n
]$ the associated posterior distribution on $S$, given by Bayes formula:
$\Pi[B| Y^n]=\int_B p_\eta^n(Y^n)\,d\Pi(\eta)/\int p_\eta
^n(Y^n)\,d\Pi
(\eta)$. Throughout the paper, we use the notation $o_p$ in the place
of $o_{P_0^n}$ for simplicity.

The quantity of interest in this paper is a functional $\psi:S\to\RR
^d, d\ge1$. We restrict in this paper to the case of real-valued
functionals $d=1$, noting that the presented tools do have natural
multivariate counterparts not detailed here for notational simplicity.

For $\eta_1,\eta_2$ in $S$, the Kullback--Leibler divergence between
$P_{\eta_1}^n$ and $P_{\eta_2}^n$ is
\[
\mathit{KL} \bigl(P_{\eta_1}^n, P_{\eta_2}^n \bigr)
:= \int_{\mathcal Y^n} \log \biggl( \frac
{ dP_{\eta_1}^n}{ dP_{\eta_2}^n }
\bigl(y^n \bigr) \biggr) \,dP_{\eta_1}^n
\bigl(y^n \bigr),
\]
and the corresponding variance of the likelihood ratio is denoted by
\[
V_n \bigl( P_{\eta_1}^n, P_{\eta_2}^n
\bigr) := \int_{\mathcal Y^n} \log^2 \biggl(
\frac{ dP_{\eta_1}^n}{ dP_{\eta_2}^n} \bigl(y^n \bigr) \biggr) \,dP_{\eta_1}^n
\bigl(y^n \bigr) -\mathit{KL} \bigl(P_{\eta_1}^n,
P_{\eta_2}^n \bigr)^2.
\]
Let
$\| \cdot\|_2$ and $\langle\cdot, \cdot\rangle_2$ denote respectively
the $L_2$ norm and the associated inner product on $[0,1]$. We use also
$\| \cdot\|_1$ to denote the $L_1$ norm on $[0,1]$.
For all $\beta\geq0$, $\mathcal C^\beta$ denotes the class of $\beta
$-H\"older functions on $[0,1]$ where $\beta=0$ corresponds to the case
of continuous functions. Let $h(f_1, f_2)=(\int_0^1 (\sqrt{f_1}-\sqrt
{f_2})^2 \,d\mu)^{1/2}$ stand for the Hellinger distance between two
densities $f_1$ and $f_2$ relative to a measure $\mu$. For $g$
integrable on $[0,1]$ with respect to Lebesgue measure, we often write
$\int_0^1 g$ or $\int g$ instead of $\int_0^1 g(x)\,dx$. For two
real-valued functions $A, B$ (defined on $\mathbb R$ or on $\mathbb
N$), we write $A\leqa B$ if $A/B$ is bounded and $A\asymp B$ if $|A/B|$
is bounded away from $0$ and $\infty$.


\section{Main result} \label{sec:main}
In this section, we give the general theorem which provides sufficient
conditions on the model, the functional and the prior for BvM to be valid.

We say that the posterior distribution for the functional $\psi(\eta)$
is \textit{asymptotically normal} with centering $\psi_n$ and
variance $V$
if, for $\beta$ the bounded Lipschitz metric (also known as the L\'
evy--Prohorov metric) for weak convergence (see Section~1 in the
supplemental article Castillo and Rousseau \cite
{castillo:rousseau:suppl}, and $\ta_n$ the
mapping $\ta_n:\eta\to\rn(\psi(\eta)-\psi_n))$, it holds, as
$n\to
\infty$, that
%
\begin{equation}
\label{def-bvmtype} \beta \bigl( \Pi \bigl[ \cdot| Y^n \bigr]\circ
\ta_n^{-1}, \cN(0,V) \bigr) \to0,
\end{equation}
in $P_0^n$-probability, which we also denote $\Pi[ \cdot|
Y^n]\circ\ta_n^{-1}
\rightsquigarrow\cN(0,V)$.

In models where an efficiency theory at rate $\rn$ is available, we say
that the posterior distribution for the functional $\psi(\eta)$ at
$\eta
=\eta_0$ \textit{satisfies the BvM theorem} if
\eqref{def-bvmtype} holds with $\psi_n=\hat\psi_n+o_p(1/\rn)$, for
$\hat
\psi_n$ a linear efficient estimator of $\psi(\eta)$ and $V$ the
efficiency bound for estimating $\psi(\eta)$. For instance, for i.i.d.
models and a differentiable functional $\psi$ with efficient influence
function $\tilde\psi_{\eta_0 }$ (see, e.g., \cite{aad98} Chapter~25),
the efficiency bound is attained if $V=P_0^n [ \tilde\psi_{\eta
_0}^2 ]$. Let us now state the assumptions which will be required.

Let $A_n$ be a sequence of measurable sets such that, as $n\to\infty$,
%
\begin{equation}
\label{defan} \Pi \bigl[ A_n| Y^n \bigr] = 1 +
o_p(1).
\end{equation}
We assume that there exists a Hilbert space $({\cal H}, \langle\cdot,
\cdot\rangle_L)$ with associated
norm denoted $\|\cdot\|_L $, and for which the inclusion $A_n- \eta_0
\subset\cal H$ is satisfied for $n$ large enough. Note that we do not
necessarily assume that $S\subset\mathcal H$, as $\mathcal H$ gives a
local description of the parameter space near $\eta_0$ only. Note also
that $\mathcal H$ may depend on $n$. The norm $\| \cdot\|_L$ typically
corresponds to the LAN (locally asymptotically normal) norm as
described in \eqref{LAN} below.

Let us first introduce some notation which corresponds to expanding
both the log-likelihood $\ell_n(\eta):=\ell_n(\eta, Y^n)$ in the model
and the functional
of interest $\psi(\eta)$. Both expansions have remainders
$R_n$ and $r$, respectively. 

\textit{LAN expansion.} 
Write, for all $\eta\in A_n$,
%
\begin{equation}
\label{LAN} \ell_n(\eta) - \ell_n(\eta_0)
= \frac{ -n \Vert \eta-\eta_0\Vert^2_L}{ 2
} + \sqrt{n}W_n(\eta- \eta_0)+
R_n(\eta,\eta_0),
\end{equation}
where [$W_n(h)$, $h \in\mathcal H$] is a collection of real random
variables verifying that, $P_0^n$-almost surely, 
the mapping $h \rightarrow W_n(h)$ is linear, and that for all $h \in
\mathcal H$, we have $W_n(h) \rightsquigarrow\cN(0,\|h\|_L^2)$ as
$n\to
\infty$.

\textit{Functional smoothness.}
Consider
$\psi_0^{(1)} \in\cal H$ and a self-adjoint linear operator $\psi
_0^{(2)}: \cal H \rightarrow\cal H$
and write, for any $\eta\in A_n$,
%
\begin{eqnarray}\label{smooth:psi}
\psi(\eta) &= & \psi( \eta_0) + \bigl\langle\psi_0^{(1)},
\eta-\eta_0 \bigr\rangle_L
\nonumber
\\[-8pt]
\\[-8pt]
\nonumber
&&{} + \tfrac{1}{2} \bigl\langle\psi_0^{(2)}(\eta-
\eta_0),\eta-\eta_0 \bigr\rangle_L+ r(
\eta, \eta_0),
\end{eqnarray}
where there exists a positive constant $C_1$ such that
%
\begin{equation}
\label{cond:dotpsi} %
\bigl\| \psi_0^{(2)} h
\bigr\|_L \leq C_1\|h\|_L\qquad \forall h \in\mathcal
H \quad\mbox{and}\quad \bigl\| \psi_0^{(1)}\bigr\|_L \leq
C_1. %
\end{equation}

Note that both formulations, on the functional smoothness and on the
LAN expansion, are not assumptions since nothing is required yet on
$r(\eta, \eta_0)$ or on $R(\eta, \eta_0)$. This is done in Assumption~\ref{assa}. The norm $\|\cdot\|_L$ is typically identified from a local
asymptotic normality property of the model at the point $\eta_0$. It is
thus intrinsic to the considered statistical model. Next, the expansion
of $\psi$ around $\eta_0$ is in term of the latter norm: since this
norm is intrinsic to the model, this can be seen as a canonical choice.

Consider two cases, depending on the value of $\psi_0^{(2)}$ in \eqref
{smooth:psi}. The first case corresponds to a first-order analysis of
the problem. It ignores any potential nonlinearity in the functional
$\eta\to\psi(\eta)$ by considering a linear approximation with
representer $\psi_0^{(1)}$ in \eqref{smooth:psi} and shifting any
remainder term into $r$.

{\textit{Case}} A1. We set $\psi_0^{(2)}=0$ in \eqref
{smooth:psi} and, for all $\eta\in A_n$ and $t \in\R$ define
%
\begin{equation}
\label{etat1} \eta_t = \eta- \frac{ t \psi_0^{(1)}}{ \sqrt{n}}.
\end{equation}

\textit{Case} A2. We allow for a nonzero second-order
term $\psi_0^{(2)}$ in \eqref{smooth:psi}. In this case, we need a few
more assumptions. One is simply the existence of some posterior
convergence rate in $\|\cdot\|_L$-norm. Suppose that, for some sequence
$\veps_n =o(1)$ and $A_n$ as in \eqref{defan},
%
\begin{equation}
\label{hyp:concentration1} \Pi \bigl[ \eta\in A_n;\| \eta- \eta_0
\|_L \leq\veps_n/2 | Y^n \bigr] = 1 +
o_p(1).
\end{equation}
Next, we assume that the action of the process $W_n$ above can be
approximated by an inner-product, with a representer $w_n$, which will
be particularly useful in defining a suitable path $\eta_t$ enabling to
handle second-order terms.

Suppose that there exists $w_n \in\mathcal H$ such that, for all $h
\in\mathcal H$,
%
\begin{equation}
\label{defwn} W_n(h) = \langle w_n, h
\rangle_L + \Delta_n(h),\qquad \mbox{$P_0^n
$-almost surely,}
\end{equation}
where the remainder term $\Delta_n(\cdot)$ is such that
%
\begin{equation}
\label{zetaAn} \sup_{\eta\in A_n } \bigl\vert\Delta_n \bigl(
\psi_0^{(2)}(\eta- \eta_0) \bigr) \bigr\vert=
o_p(1)
\end{equation}
%
and where one further assumes that
%
\begin{equation}
\label{zetapsi} \bigl\langle w_n, \psi_0^{(2)}
\bigl( \psi_0^{(1)}\bigr) \bigr\rangle_L +
\veps_n \| w_n\|_L = o_p(
\sqrt{n}).
\end{equation}
Finally set, for all $\eta\in A_n$ and $w_n$ as in \eqref{defwn}, for
all $t \in\R$,
%
\begin{equation}
\label{etat2} \eta_t = \eta- \frac{ t \psi_0^{(1)}}{ \sqrt{n}} -
\frac{ t \psi_0^{(2)}
(\eta
- \eta_0) }{ 2 \sqrt{n}} - \frac{ t \psi_0^{(2)}w_n}{ 2n }.
\end{equation}


\renewcommand{\theassumption}{\Alph{assumption}}
\begin{assumption}\label{assa}
In cases {A1} and {A2}, with $\eta_t$ defined by
\eqref{etat1}
and \eqref{etat2}, respectively, assume that for all $t\in\R$, $\eta_t
\in S$ for $n$ large enough and that
%
\begin{equation}
\label{Rn} \sup_{\eta\in A_n }\bigl|t \sqrt{n} r (\eta,
\eta_0) + R_n(\eta,\eta_0) -
R_n(\eta_t, \eta_0)\bigr| = o_p(1).
\end{equation}
\end{assumption}

The suprema in the previous display may not be measurable, in this case
one interprets the previous probability statements in
terms of outer measure.


We then provide a characterisation of the asymptotic distribution of
$\psi(\eta)$. At first read, one may set $\psi_0^{(2)}=0$ in the next
theorem: this provides a first-order result that will be used
repeatedly in Sections \ref{sec:density} and \ref{sec:autoreg}. The
complete statement allows for a second-order analysis via a possibly
nonzero $\psi_0^{(2)}$ and will be applied in Section~\ref{sec:WNmodels}.

\begin{thmm} \label{lem:laplace}
Consider a statistical model $\{P_\eta^n, \eta\in S\}$, a real-valued
functional $\eta\to\psi(\eta)$ and $\langle\cdot, \cdot\rangle_L,
\psi
_0^{(1)}, \psi_0^{(2)}, W_n, w_n$ as defined above.
Suppose that Assumption~\ref{assa} is satisfied, and denote
\[
\hat{\psi} = \psi(\eta_0) + \frac{W_n(\psi_0^{(1)})}{ \sqrt{n} } + \frac
{ \langle w_n, \psi_0^{(2)} w_n\rangle_L}{ 2 n },\qquad
V_{0,n} = \biggl\Vert\psi_0^{(1)} - \frac{ \psi_0^{(2)}w_n }{ 2 \sqrt
{n} }
\biggr\Vert_L^2.
\]
Let $\Pi$ be a prior distribution on $\eta$.
Let $A_n$ be any measurable set such that \eqref{defan} holds. Then for
any real $t$ with $\eta_t$ as in
\eqref{etat2},
%
\begin{equation}\quad
\label{laplace} 
E^\Pi \bigl[ e^{t \sqrt{n} ( \psi(\eta) - \hat{\psi} )} |
Y^n, A_n \bigr] =e^{o_p(1) + {t^2V_{0,n} }/{ 2 }} \frac{ \int_{A_n} e^{\ell_n(\eta_t) - \ell_n(\eta_0) } \,d\Pi(\eta
) }{
\int_{A_n} e^{\ell_n(\eta) - \ell_n(\eta_0) } \,d\Pi(\eta)}.
\end{equation}
%
Moreover, if $V_{0,n}=V_0+o_p(1)$ for some $V_0>0$ and if for some
possibly random sequence of reals $\mu_n$, for any real $t$,
%
\begin{equation}
\label{big:cond} \frac{ \int_{A_n} e^{\ell_n(\eta_t) - \ell
_n(\eta_0) } \,d\Pi(\eta
) }{
\int_{A_n} e^{\ell_n(\eta) - \ell_n(\eta_0) } \,d\Pi(\eta)} = e^{\mu_n t} \bigl(1 +
o_p(1) \bigr),
\end{equation}
then the posterior distribution of $\rn(\psi(\eta)-\hat{\psi})-\mu_n$
is asymptotically normal and mean-zero, with variance $V_{0}$.
\end{thmm}

The proof of Theorem~\ref{lem:laplace} is given in Section~\ref
{sec:pr:lem:laplace}.

\begin{cor}
Under the conditions of Theorem~\ref{lem:laplace}, if \eqref{big:cond}
holds with $\mu_n = o_p(1)$ and $\|\psi_0^{(2)} w_n\|_L = o_p(\sqrt
{n})$, then
the posterior distribution of  $\rn(\psi(\eta)-\hat\psi)$ is
asymptotically mean-zero normal, with variance
$\|\psi_0^{(1)}\|_L^2$.
\end{cor}

Assumption~\ref{assa} ensures that the local behaviour of the
likelihood resembles
the one in a Gaussian experiment with norm $\|\cdot\|_L$. An assumption
of this type is expected, as the target distribution in the BvM theorem
is Gaussian. As will be seen in the examples in Sections \ref
{sec:WNmodels}, \ref{sec:density} and \ref{sec:autoreg}, $A_n$ is often
a well chosen subset of a neighbourhood of $\eta_0$, with respect to a
given metric, which need not be the LAN norm $\| \cdot\|_L$.

We note that for simplicity here we restrict to approximating paths
$\eta_t$ to $\eta_0$ in~\eqref{etat1} (first-order results) and
\eqref
{etat2} (second-order results) that are linear in the perturbation.
This covers already quite a few interesting models. More generally,
some models may be locally \emph{curved} around $\eta_0$, with a
possibly nonlinear form of approximating paths. A more general
statement would possibly have an extra condition to control the
curvature. Examining this type of example is left for future work.


The central condition for applying Theorem~\ref{lem:laplace} is \eqref
{laplace}. To check this condition, a possible approach is to construct
a change of parameter from $\eta$ to $\eta_t$ (or some parameter close
enough to $\eta_t$), which leaves the prior and $A_n$ approximately
unchanged. More formally, let $\psi_n$ be an approximation of $\psi
_0^{(1)}$ in a sense to be made precise below and let $\Pi^{\psi
_n}:=\Pi
\circ(\ta^{\psi_n})^{-1}$ denote the image measure of $\Pi$ through
the mapping
\[
\ta^{\psi_n}:\eta\to\eta-t\psi_n/\rn.
\]
To check \eqref{laplace}, one may for instance suppose that the
measures $\Pi^{\psi_n}$ and $\Pi$ are mutually absolutely continuous
and that the density $d\Pi/d\Pi^{\psi_n}$ is close to the quantity
$e^{\mu_n t}$ on $A_n$.
This is the approach we follow for various models and priors in the
sequel. In particular, we prove that a functional change of variable is
possible for various classes of prior distributions. For instance, in
density estimation, Gaussian process priors and piecewise constant
priors are considered and Propositions~\ref{prop:bvm:hist} and~\ref
{thmm-gp} below give a set of sufficient conditions that guarantee
\eqref
{laplace} for each class of priors.

In general, the construction of a \textit{feasible} change of
parameterisation heavily depends on the structure of the prior model.
We note that this change of parameter approach above only provides a
sufficient condition. For some priors, shifted measures may be far from
being absolutely continuous, even using approximations of the shifting
direction: for such priors, one may have to compare the integrals directly.

\begin{rem}
Here, the main focus is on estimation of abstract semiparametric
functionals $\psi(\eta)$. Our results also apply to the case of
separated semiparametric models where $\eta= (\psi,f)$ and $\psi
(\eta
)=\psi\in\R$, as considered in \cite{icbvm}, with a weak convergence
to the normal distribution instead of a strong convergence obtained in
\cite{icbvm}. We have $\psi(\eta)-\psi(\eta_0) = \langle\eta
-\eta
_0,(1,-\ga) \rangle_L/\tilde I_{\eta_0}$ where $\gamma$ is the least
favorable direction and $\tilde I_{\eta_0}=\|(1,-\ga)\|_L^2$; see
\cite{icbvm}.
We can then choose $\psi_0^{(1)} = (1,-\ga) / \tilde I_{\eta_0}$ in
\cite{icbvm}. If $\gamma= 0$ (no loss of information),
$\eta_t = (\psi- t\tilde I_{\eta_0}^{-1}/\sqrt{n}, f)$ and \eqref
{laplace} is satisfied if $\pi= \pi_\psi\otimes\pi_f $ with $\pi
_\psi$
positive and continuous at $\psi(\eta_0)$, so that we obtain a similar
result as Theorem~1 of \cite{icbvm}. In \cite{icbvm}, a slightly weaker
version of condition \eqref{Rn} is considered; however, the proof of
Section~\ref{sec:pr:lem:laplace} can be easily adapted---in the case
of separated semiparametric models---so that the result holds under
the weaker version of \eqref{Rn} as well.
\end{rem}

\begin{rem}\label{com:psiu}
As follows from the proof of Theorem~\ref{lem:laplace}, $\psi
_0^{(1)}$ can be
replaced by any element, say $\tilde\psi$ of $\mathcal H$ such that
\[
\langle\tilde\psi, \eta- \eta_0 \rangle_L = \bigl
\langle \psi^{(1)}_0, \eta- \eta_0 \bigr
\rangle_L,\qquad \| \tilde\psi\|_L = \bigl\|
\psi_0^{(1)} \bigr\|_L,
\]
where $\tilde\psi$ may potentially depend on $\eta$. This proves to be
useful when considering constraint spaces as in the case of density estimation.
\end{rem}


We now apply Theorem~\ref{lem:laplace} in the cases of white noise,
density and autoregressive models and for various types of functionals
and priors.


\section{Applications to the white noise model} \label{sec:WNmodels}
Consider the model
\[
dY^n(t) = f(t)\,dt + n^{-1/2} \,dB(t),\qquad t \in[0,1],
\]
where $f\in L^2[0,1]$ and $B$ is standard Brownian motion. Let $(\phi
_k)_{k\ge1}$ be an orthonormal basis for $L^2[0,1]=:L^2$. The model
can be rewritten
\[
Y_{k} = f_k + n^{-1/2}\epsilon_k,\qquad
f_k = \int_0^1 f(t)
\phi_k(t) \,dt,\qquad \ep_k \sim\mathcal N(0,1)\qquad \mbox{i.i.d.}, k
\ge1.
\]
The likelihood admits a LAN expansion, with $\eta=f$ here, $\|\cdot\|
_L=\|\cdot\|_2$ and $R_n = 0$:
\[
\ell_n(f) - \ell_n(f_0) = -
\frac{n \| f - f_0\|^2 }{2 } +\sqrt{n} W(f- f_0),
\]
where for any $u\in L^2 = \mathcal H$ with coefficients $u_k=\int_0^1
u(t)\phi_k(t)\,dt$, we set $W(u) = \sum_{k \ge1}\epsilon_k u_k $.

In this model, consider the squared-$L^2$ norm as a functional of $f$. Set
\begin{eqnarray*}
\psi(f) &=& \| f\|_2^2 = \psi(f_0) + 2
\langle f_0, f - f_0\rangle_2 + \| f -
f_0\|_2^2,
\\
\psi_0^{(1)} &=& 2 f_0,\qquad \psi_0^{(2)}
h = 2 h, \qquad r(f,f_0)=0.
\end{eqnarray*}
The functional has been extensively studied in the frequentist
literature; see \cite{BR88,efromovitch:low:96,L96,GT99} and \cite{cai:low:2006} to name but a few, as it is used in
many testing problems.
The verification of Assumption~\ref{assa} and of condition \eqref
{big:cond} is prior-dependent and is considered within the proof of the
next theorem.

Suppose that the true function $f_0$ belongs to the Sobolev class
\[
W_{\beta}:= \biggl\{ f\in L^2, \sum
_{k\ge1} k^{2\beta} \langle f,\phi_k
\rangle^2 <\infty \biggr\}
\]
of order $\beta>1/4$. First, one should note that, while the case
$\beta
>1/2$ can be treated using the first-order term of the expansion of the
functional only (case A1), the case $1/4 < \beta< 1/2$
requires the conditions from case A2 as the second-order term
cannot be neglected. This is related to the fact that the so-called
plug-in property in \cite{BR03} does not work for $\beta<1/2$. An
analysis based on second-order terms as in Theorem~\ref{lem:laplace} is
thus required.
The case $\beta\le1/4$ is interesting too, but one obtains a rate
slower than $1/\sqrt{n}$; see, for example, Cai and Low \cite
{cai:low:2006} and
references therein, and a BvM result in a strict sense does not hold.
Although a BvM-type result can be obtained essentially with the tools
developed here, its formulation is more complicated and this case will
be treated elsewhere.
When $\beta> 1/4$, a natural frequentist estimator of $\psi(\eta)$ is
\[
\bar\psi:=\bar\psi_n:= \sum_{k=1}^{K_n}
\biggl[ Y_k^2 - \frac
{1}{n} \biggr] \qquad\mbox{with }
K_n=\lfloor n/\log{n}\rfloor.
\]

Now define a prior $\Pi$ on $f$ by sampling independently
each coordinate $f_k$,
$k\ge1$ in the following way. Given a density $\varphi$ on $\RR$ and a
sequence of positive real numbers $(\sigma_k)$, set $K_n=\lfloor
n/\log
{n}\rfloor$ and
%
\begin{equation}
\label{gw-prior} f_k \sim\frac{1}{\sigma_k}\vphi \biggl(
\frac{\cdot}{\sigma_k} \biggr) \qquad\mbox{if } 1\le k\le K_n\quad \mbox {and}\quad
f_k = 0 \qquad\mbox{if } k>K_n.
\end{equation}
In particular, we focus on the cases where $\varphi$ is either the
standard Gaussian density or $\varphi(x)=\1_{[-\cM,\cM]}(x)/\cM$,
$\cM
>0$, called respectively \emph{Gaussian} $\varphi$ and \emph{Uniform}~$\varphi$.

Suppose that there exists $M>0$ such that, for any $1\le k \le K_n$, 
%
\begin{equation}
 \frac{|f_{0,k}|}{\sigma_k} \le M\quad \mbox{and}\quad \sigma_k\ge\frac
{1}{\sqrt{n}}.
\label{wn-cond-sig} 
\end{equation}
%

\begin{thmm}\label{thmm-wn}
Suppose the true function $f_0$ belongs to the Sobolev space $W_{\beta
}$ of order $\beta>1/4$. Let the prior $\Pi$ and $K_n$ be chosen
according to \eqref{gw-prior} and let $f_0, \{\sigma_k\}$ satisfy
\eqref
{wn-cond-sig}.
Consider the following choices for $\varphi$:
\begin{longlist}[1.]
\item[1.] Gaussian $\varphi$. Suppose that as $n\to\infty$,
%
\begin{equation}
\label{wn-cond-gauss} \frac{1}{\sqrt{n}} \sum_{k=1}^{K_n}
\frac{\sigma_k^{-2}}{n} = o(1).
\end{equation}
\item[2.] Uniform $\varphi$. Suppose $\cM>4 \vee(16M)$ and that for any
$c>0$
%
\begin{equation}
\label{wn-cond-unif} \sum_{k=1}^{K_n}
\sigma_k e^{-cn\sigma_k^2} = o(1).
\end{equation}
%
\end{longlist}
Then, in $P_{f_0}^n$-probability, as $n\to\infty$,
%
\begin{equation}
\label{bvm-l2} \Pi \biggl(\rn \biggl(\psi(f)-\bar\psi-2\frac{K_n}{n}
\biggr) \Big| Y^n \biggr) \rightsquigarrow\cN \bigl(0,4 \|f_0
\|^2_2 \bigr).
\end{equation}
\end{thmm}

The proof of Theorem~\ref{thmm-wn} is given in Section~2.2 of the
supplemental article Castillo and Rousseau \cite
{castillo:rousseau:suppl}. 

Theorem~\ref{thmm-wn} gives the BvM theorem for the nonlinear
functional $\psi(f)=\int f^2$, up to a (known) bias term $2K_n/n$.
Indeed it implies that the posterior distribution of $\psi(f) - \hat
\psi_n = \psi(f) - \bar\psi-2\frac{K_n}{n}$ is asymptotically
Gaussian with mean 0 and variance $4\|f_0\|^2_2/n$ which is the inverse
of the efficient information\vspace*{1pt} (divided by $n$). Recall that $\bar\psi$
is an efficient estimate when $\beta> 1/4$; see, for instance, \cite
{cai:low:2006}. Therefore, even though the posterior distribution of
$\psi(\eta)$ does not satisfy the BvM theorem per se, it can be
modified a posteriori by recentering with the known quantity $2K_n/n$
to lead to a BvM theorem.
The possibility of existence of a Bayesian nonparametric prior leading
to a BvM for the functional $\|f\|_2^2$ without any bias term in
general is unclear.
However, if we restrict our attention to $\beta>1/2$, a different
choice of $K_n$ can be made, in particular $K_n=\sqrt{n}/\log{n}$ leads
to a standard BvM property without bias term.

Condition \eqref{wn-cond-sig} can be interpreted as an undersmoothing
condition: the true function should be at least as ``smooth'' as the
prior; for a fixed prior, it corresponds to intersecting the Sobolev
regularity constraint on $f_0$ with a H\"older-type constraint. It is
used to verify the concentration of the posterior \eqref
{hyp:concentration1}; see Lemma~3 of the supplemental article
(Castillo and Rousseau \cite
{castillo:rousseau:suppl}) (it is used here mostly for simplicity of
presentation and can possibly be slightly improved).
For instance, if $\sigma_k \gtrsim k^{-1/4} $ for all $k \leq K_n$,
then condition \eqref{wn-cond-sig} is valid for all $f_0 \in W_\beta$,
with $\beta> 1/4$.
Conditions \eqref{wn-cond-gauss} and \eqref{wn-cond-unif} are here to
ensure that the prior is hardly modified by the change of
parametrisation \eqref{etat2}, they are verified in particular for any
$\sigma_k \gtrsim k^{-1/4}$.

An interesting phenomenon appears when comparing the two examples of
priors considered in Theorem~\ref{thmm-wn}. If $\sigma_k=k^{-\delta}$, for some $\delta\in\R$,
condition \eqref{wn-cond-gauss} holds for any $\delta\leq1/4$ in the
Gaussian $\varphi$ case, whereas \eqref{wn-cond-unif} only requires
$\delta<1/2$ in the Uniform $\varphi$ case, this for any $f_0$ in
$W_{1/4}$ intersected with the H\"{o}lder-type space $\{f_0:
|f_{0,k}|\le M k^{-\delta}, k\ge1\}$. 
One can conclude that \textit{fine details} of the prior (here, the
specific form of $\varphi$ chosen, for given variances $\{\sigma_k^2\}
$) really matter for BvM to hold in this case. Indeed, it can be
checked that the condition for the Gaussian prior is sharp: while the
proof of Theorem~\ref{thmm-wn} is an application of the general Theorem~\ref{lem:laplace}, a completely different proof can be given for
Gaussian priors using conjugacy, similar in spirit to \cite{kvv11},
leading to \eqref{wn-cond-gauss} as a necessary condition.
Hence, choosing $\sigma_k \gtrsim k^{-1/4}$ leads to a posterior
distribution satisfying the BvM property adaptively over Sobolev balls
with smoothness $\beta>1/4$.

The introduced methodology also allows us to provide conditions under
generic smoothness assumptions on $\varphi$. For instance, if the
density $\varphi$ of the prior is a Lipschitz function on $\RR$, then
the conclusion of Theorem~\ref{thmm-wn} holds when, as $n\to\infty$,
%
\begin{equation}
\label{wn-cond-lip} \sum_{k=1}^{K_n}
\frac{\sigma_k^{-1}}{n} = o(1).
\end{equation}
This last condition is not sharp in general [compare for instance with
the sharp \eqref{wn-cond-gauss} in the Gaussian case], but provides a
sufficient condition for a variety of prior distributions, including
light and heavy tails behaviours. For instance, if $\sigma
_k=k^{-\delta
}$, then \eqref{wn-cond-lip} asks for $\delta\leq0$.




\section{Application to the density model} \label{sec:density}

The case of functionals of the density is another interesting
application of Theorem~\ref{lem:laplace}. The case of linear
functionals of the density has first been considered by \cite
{rivoirard:rousseau:09}. Here, we obtain a broader version of Theorem~2.1 in \cite{rivoirard:rousseau:09}, which weakens the assumptions for
the case of linear functionals and also allows for nonlinear functionals.

\subsection{Statement}
Let $Y^n = (Y_1,\ldots,Y_n)$ be independent and identically
distributed, having density $f$ with respect to Lebesgue measure on the
interval $[0,1]$. In all of this section, we assume that the true
density $f_0$ belongs to the set $\cF_0$ of all densities that are
bounded away from $0$ and $\infty$ on $[0,1]$. Let us consider $A_n =
\{ f ; \| f - f_0 \|_1 \leq\veps_n\}$ where $\veps_n$ is a sequence
decreasing to 0, or any set of the form $A_n \cap\cF_n$, as long as
$P_0^n\Pi(\cF_n^c| Y^{n})
\to0$.
Define
\[
L^2(f_0) = \biggl\{ \vphi:[0,1]\to\RR, \int
_0^1 \vphi(x)^2f_0(x)\,dx
<\infty \biggr\}.
\]
For any $\vphi$ in $L^2(f_0)$, let us write $F_0( \vphi)$ as shorthand
for $\int_0^1 \vphi(x)f_0(x)\,dx$ 
and set, for any positive density $f$ on $[0,1]$,
\[
\eta=\log f,\qquad \eta_0=\log f_0,\qquad h=\rn(\eta-
\eta_0).
\]
Following \cite{rivoirard:rousseau:09}, we have the LAN expansion
\begin{eqnarray*}
\ell_n(\eta)-\ell_n(\eta_0) & =& \rn
F_0(h) + \frac{1}{\rn} \sum_{i=1}^n
\bigl[ h(Y_i) - F_0(h) \bigr]
\\
& =& -\frac{1}{2} \| h \|_L^2 +
W_n(h) + R_n(\eta,\eta_0),
\end{eqnarray*}
with the following notation, for any $g$ in $L^2(f_0)$,
%
\begin{equation}\quad
\label{dens-def} \| g \|_L^2 = \int_0^1
\bigl(g - F_0( g) \bigr)^2 f_0,
W_n(g) =\mathbb{G}_n g= \frac{1} \rn\sum
_{i=1}^n \bigl[g(Y_i)-F_0(g)
\bigr],
\end{equation}
and $R_n(\eta,\eta_0)=\rn P_{f_0}h + \frac{1}{2} \| h \|_L^2$. Note that
$\|\cdot\|_L$ is an Hilbertian norm induced by the inner-product
$\langle
g_1,g_2 \rangle_L =\int g_1 g_2 f_0 $ defined on the space $\mathcal
H_T:=\{g\in L^2(P_{f_0}), \int g f_0 =0\} \subset\mathcal H = L^2(f_0)$,
the so-called maximal tangent set at $f_0$.

We consider functionals $\psi(f)$ of the density $f$, which are
differentiable relative to (a dense subset of) the tangent set
$\mathcal H_T$ with efficient influence function $\tilde{\psi
}_{f_0}$; see \cite
{aad98}, Chapter~25. In particular, $\tilde{\psi}_{f_0}$ belongs to
$\mathcal H_T$,
so $F_0(\tilde{\psi}_{f_0}) = 0$. We further assume that $\tilde
{\psi}_{f_0}$ is \textit
{bounded} on
$[0,1]$. Set
%
\begin{eqnarray}\label{func}
&&\psi(f) -\psi(f_0)\nonumber\\
&&\qquad= \biggl\langle\frac{f-f_0}{f_0} , \tilde{
\psi}_{f_0} \biggr\rangle_L + \tilde r(f,f_0)
\\
&&\qquad = \bigl\langle\eta- \eta_0 - F_0(\eta-
\eta_0) , \tilde{\psi}_{f_0}\bigr\rangle_L +
\mathcal{B}(f,f_0)+ \tilde r(f,f_0),\qquad \eta= \log f,\nonumber
\end{eqnarray}
where $\cB(f,f_0)$ is the difference
\[
\mathcal B(f,f_0) = \int_0^1
\biggl[ \eta-\eta_0-\frac{f-f_0}{f_0} \biggr](x) \tilde{
\psi}_{f_0}(x) f_0(x) \,dx,
\]
%
and define $r(f,f_0) = \mathcal{B}(f,f_0)+ \tilde r(f,f_0)$.

\begin{thmm} \label{thmm-fundensity}
Let $\psi$ be a differentiable functional relative to the tangent set
$\cH_T$, with efficient influence function $\tilde{\psi}_{f_0}$
bounded on $[0,1]$.
Let $\tilde r$ be defined by~\eqref{func}.
Suppose that for some $\veps_n\to0$ it holds
%
\begin{equation}
\label{cond:conc:dens} \Pi \bigl[ f: \|f-f_0\|_1 \le
\veps_n | Y^{n} \bigr] \to1,
\end{equation}
%
in $P_{0}$-probability and that, for $A_n=\{ f, \|f - f_0\|_1 \le\veps
_n\}$,
\[
\sup_{f \in A_n} \tilde r(f,f_0) = o(1/\sqrt{n}).
\]
Set $\eta_t=\eta-\frac{t}{\rn}\tilde{\psi}_{f_0}-\log\int_0^1
e^{\eta-({t}/{\rn})\tilde{\psi}_{f_0}}$
and assume that in $P_0$-probability
%
\begin{equation}
\label{bigcond:dens} \frac{\int_{A_n} e^{\ell_n(\eta_t)-\ell
_n(\eta_0) } \,d\Pi(\eta)}{
\int e^{\ell_n(\eta)-\ell_n(\eta_0)} \,d\Pi(\eta)} \to1.
\end{equation}
%
Then, for $\hat{\psi}$ any linear efficient estimator of $\psi(f)$,
the BvM
theorem holds for the functional $\psi$. That is, the posterior
distribution of $\rn(\psi(f)-\hat{\psi})$ is asymptotically
Gaussian with
mean 0 and variance $\|\tilde{\psi}_{f_0}\|_L^2$, in $P_{0}$-probability.
\end{thmm}

The semiparametric efficiency bound for estimating $\psi$ is $\|\tilde
{\psi}_{f_0}
\|
_L^2$ and linear efficient estimators of $\psi$ are those for which
$\hat{\psi}= \psi(f_0) + \mathbb G_n( \tilde{\psi}_{f_0}) / \sqrt
{n} + o_p( 1/\sqrt
{n})$; see, for example, van der Vaart \cite{aad98}, Chapter~25, so
Theorem~\ref
{thmm-fundensity} yields the BvM theorem (with best possible limit
distribution). 

\begin{rem} \label{rem-h}
The $L^1$-distance between densities in Theorem~\ref{thmm-fundensity}
can be replaced by Hellinger's distance $h(\cdot, \cdot)$ up to
replacing $\veps_n$ by $\veps_n/\sqrt{2}$. 
\end{rem}

Theorem~\ref{thmm-fundensity} is proved in Section~\ref{sec:proofs} and
is deduced from Theorem~\ref{lem:laplace} with $\psi_0^{(2)}= 0$ and
$\psi_0^{(1)} = \tilde{\psi}_{f_0}- t^{-1}\sqrt{n}\log\int_0^1
e^{\eta-({t}/{\rn})\tilde{\psi}_{f_0}}$. The condition\break $ \sup_{f \in A_n} \tilde
r(f,f_0) =
o(1/\sqrt
{n})$, together\vspace*{1pt} with \eqref{cond:conc:dens} imply Assumption~\ref{assa}.
It improves on Theorem~2.1 of \cite{rivoirard:rousseau:09} in the sense
that an $L_1$ -posterior concentration rate is required instead of a
posterior concentration rate in terms of the LAN norm $\|\cdot\|_L$,
it is also a generalisation to approximately linear functionals, which
include the following examples.

\begin{example}[(Linear functionals)] \label{exa-lin}
Let $\psi(f)=\int_0^1 f(x)a(x)\,dx$, for some bounded function $a$. Then,
writing $\int$ as shorthand for $\int_0^1$,
\[
\psi(f) -\psi(f_0) = \biggl\langle\frac{f-f_0}{f_0} , a - \int a
f_0 \biggr\rangle_L
\]
with the efficient influence function $\tilde{\psi}_{f_0}= a - \int a
f_0 $. In this
case, $\tilde r(f,f_0)=0$.
\end{example}

\begin{example}[(Entropy functional)] \label{exa-ent}
Let $\psi(f)=\int_0^1 f(x)\log f(x)\,dx$, for $f$ bounded away from $0$
and infinity. Then
\[
\psi(f) -\psi(f_0) = \biggl\langle\frac{f-f_0}{f_0} , \log
f_0 - \int f_0\log{f_0} \biggr
\rangle_L + \int f\log \frac{f} {f_0}
\]
with the efficient influence function $\tilde{\psi}_{f_0}= \log f_0-
\int f_0 \log
{f_0}$. In this case,
$\tilde r(f,f_0)=\int f\log\frac{f} {f_0}$. For the two types of priors
considered below, under some smoothness assumptions on $f_0$, it holds
$\sup_{f \in A_n} \tilde r(f,f_0)= o(1/\sqrt{n})$.
\end{example}

\begin{example}[(Square-root functional)] \label{exa-sqr}
Let $\psi(f)=\int_0^1 \sqrt{f(x)}\,dx$, for $f$ a bounded density. Then
\[
\psi(f) -\psi(f_0) = \frac{1}{2} \biggl\langle
\frac{f-f_0}{f_0} , \frac{
1}{ \sqrt{f_0}} - \int\sqrt{f_0} \biggr
\rangle_L + \frac{1}{2} \int\frac{ \sqrt{f_0}-\sqrt{f} }{\sqrt{f_0}+\sqrt
{f}}
\frac
{f-f_0}{\sqrt{f_0}}
\]
with the efficient influence function $\tilde{\psi}_{f_0}=\frac{1} 2
(\frac{ 1}{
\sqrt{f_0}}
- \int\sqrt{f_0} )$.
In this case, $\tilde r(f,f_0) = - \int\frac{(\sqrt{f_0}-\sqrt
{f})^2}{2\sqrt{f_0}}$. In particular, the remainder term of the
functional expansion is bounded by a constant times the square of the
Hellinger distance between densities, hence as soon as $\veps_n^2
\sqrt
{n} = o(1)$, if $A_n$ is written in terms of $h$ (see Remark~\ref
{rem-h}), one has $\sup_{f \in A_n} \tilde r(f,f_0)= o(1/\sqrt{n})$.
\end{example}

\begin{example}[(Power functional)] \label{exa-pow}
Let $\psi(f)=\int_0^1 f(x)^q \,dx$, for $f$ a bounded density and $q\ge
2$ an integer. Then
\[
\psi(f) -\psi(f_0) = \biggl\langle\frac{f-f_0}{f_0} ,
qf_0^{q-1}-q \int f_0^q \biggr
\rangle_L + r(f,f_0).
\]
The remainder $\tilde r(f,f_0)$ is a sum of terms of the form $\int
(f-f_0)^{2+s}f_0^{q-2-s}$, for $0\le s\le q-2$ an integer.
For the two types of priors considered below, $\sup_{f \in A_n} \tilde
r(f,f_0)= o(1/\sqrt{n})$, under some smoothness assumptions on $f_0$.
\end{example}

We now consider two families of priors: random histograms and Gaussian
process priors. For each family, we provide a key no-bias condition for
BvM on functionals to be valid. For each, the idea is based on a
certain functional change of variables formula.
To simplify the notation, we write $\tilde\psi=\tilde\psi_{f_0}$ in
the sequel.

\subsection{Random histograms} \label{subsec:dens:hist}

For any $k \in\N^*$, consider the partition of $[0,1]$ defined by
$I_j = [(j-1)/k, j/k)$ for $j =1,\ldots, k$.
Denote by
\[
\cH_k = \Biggl\{g\in L^2[0,1], g(x) = \sum
_{j=1}^k g_j \1_{I_j}(x),
g_j\in\mathbb{R}, j=1,\ldots,k \Biggr\}
\]
the set of all regular histograms with $k$ bins on $[0,1]$. Let
$\mathcal S_k = \{ \omega\in[0,1]^k;\break  \sum_{j=1}^k\omega_j =1\}$ be
the unit simplex in $\R^k$ and denote $\cH_k^1$ the subset of $\cH_k$
consisting of histograms which are densities on $[0,1]$:
\[
\cH_k^1= \Biggl\{f\in L^2[0,1], f(x) =
f_{\omega,k}= k \sum_{j=1}^k
\omega_j \1_{I_j}(x), (\omega_1,\ldots,
\omega_k)\in\cS_k \Biggr\}.
\]
A prior on $\cH_k^1$ is completely specified by the distributions of
$k$ and of $(\omega_1,\ldots,\omega_k)$ given $k$.
Conditionally, on $k$, we consider a Dirichlet prior on $\omega=
(\omega_1,\ldots,\omega_k)$:
%
\begin{equation}
\label{pr:omk} \omega\sim\cD(\alpha_{1,k}, \ldots,
\alpha_{k,k}),\qquad c_1 k^{-a}\le\alpha_{j,k}
\le c_2,
\end{equation}
for some fixed constants $a, c_1,c_2>0$ and any $1\le j\le k$.

Consider two situations: either a deterministic number of bins with $k
= K_n=o(n)$
or, for $\pi_k$ a distribution on positive integers,
%
\begin{equation}
\label{prrk} k \sim\pi_k,\qquad e^{-b_1k \log(k) } \leq
\pi_k(k) \leq e^{-b_2 k
\log(k)},
\end{equation}
for all $k$ large enough and some $0<b_2 < b_1< \infty$. Condition
\eqref{prrk} is verified for instance by the Poisson distribution which
is commonly used in Bayesian nonparametric models; see, for instance,
\cite{arbeletal:13}. 

The set $\cH_k$ is a closed subspace of $L^2[0,1]$. For any function
$h$ in $L^2[0,1]$, consider its projection $h_{[k]}$ in the $L^2$-sense
on $\cH_k$.
It holds
\[
h_{[k]} = k \sum_{j=1}^k
\biggl\{ \int_{I_j} h \biggr\} \1_{I_j}.
\]
%
Lemma~4 in the supplemental article Castillo and Rousseau \cite
{castillo:rousseau:suppl}
gathers useful properties on histograms.


Let the functional $\psi$ satisfy \eqref{func} with bounded efficient
influence function $\tilde{\psi}_{f_0}=\tilde{\psi}\neq0$ and set,
for $k\ge1$,
%
\begin{eqnarray}
\label{def-kfixe} %
\hat\psi_k & = &\psi(f_{0[k]}) +
\frac{ \mathbb G_n \tilde\psi_{[k]}
}{\sqrt{n}}, \qquad V_k = \| \tilde\psi_{[k]}
\|_L^2,
\nonumber
\\[-8pt]
\\[-8pt]
\nonumber
\hat\psi& =& \psi(f_{0}) + \frac{ \mathbb G_n \tilde\psi}{\sqrt
{n}},\qquad V = \| \tilde\psi
\|_L^2, %
\end{eqnarray}
with $\|\cdot\|_L, \mathbb G_n$ as in \eqref{dens-def}.
Finally, for $n\ge2$, $k\ge1$, $M>0$, denote
%
\begin{equation}
\label{def-vepan} A_{n,k}(M)= \bigl\{f\in\cH_k^1,
h(f,f_{0,[k]})\le M\veps_{n,k} \bigr\} \qquad \mbox{with }
\veps_{n,k}^2= \frac{k \log n}{n}.
\end{equation}

In Section~\ref{sec:pr:bvm:hist}, we shall see that the posterior
distribution of $k$ concentrates on a deterministic subset $\mathcal
K_n$ of $\{ 1, \ldots, \lfloor n/(\log n)^2 \rfloor\}$ and that under
the following technical condition on the weights, as $n\to\infty$,
%
\begin{equation}
\label{cond-dir} \sup_{k \in\mathcal K_n} \sum_{j=1}^k
\al_{j,k} = o( \sqrt{n} ), %
\end{equation}
the conditional posterior distribution given $k$, concentrates on the
sets $A_{n,k}(M)$. It can then be checked that
%
\begin{eqnarray*}
& &\Pi \bigl[ \sqrt{n} ( \psi- \hat\psi) \leq z |Y^n \bigr]
\\
&&\qquad= \sum_{k\in\cK_n} \Pi \bigl[k| Y^n \bigr]
\Pi \bigl[ \sqrt{n} ( \psi- \hat\psi_k) \leq z + \sqrt{n}( \hat\psi-
\hat\psi_k) |Y^n,k \bigr] +o_p(1)
\\
&&\qquad= \sum_{k\in\cK_n} \Pi \bigl[k| Y^n \bigr]
\Phi \bigl( \bigl( z + \sqrt{n}( \hat\psi- \hat\psi _k) \bigr)/
\sqrt{V_k} \bigr) + o_p(1).
\end{eqnarray*}
%
The last line expresses that the posterior is asymptotically close to a
mixture of normals, and that the mixture reduces to the target law
$N(0,V)$ if $V_k$ goes to $V$ and $\rn(\hat\psi- \hat\psi_k)$ to
$0$, uniformly for $k$ in $\cK_n$. The last quantity can also be rewritten
%
\begin{eqnarray*}
\sqrt{n}(\hat\psi_{k}-\hat\psi) & =& \rn \bigl(
\psi(f_{0[k]})-\psi(f_0) \bigr) + \mathbb{G}_n (
\tilde\psi_{[k]}-\tilde\psi)
\\
& =& \rn\int(\tilde\psi-\tilde\psi_{[k]}) (f_{0[k]}-f_0)
+ \mathbb{G}_n (\tilde\psi_{[k]}-\tilde\psi) + o(1).
\end{eqnarray*}
%
It is thus natural to ask for, and this is satisfied in most examples
(see below),
%
\begin{equation}\qquad
\label{csimp} \max_{k\in\cK_n} \bigl|\|\tilde\psi_{[k]}
\|_L^2-\|\tilde\psi\|_L^2 \bigr|
=o_p(1)\quad \mbox{and}\quad \max_{k\in\cK_n}
\mathbb{G}_n (\tilde\psi_{[k]}-\tilde\psi) =
o_p(1).
\end{equation}
This leads to the next proposition, proved in Section~\ref{sec:proofs}.

\begin{prop} \label{prop:bvm:hist} 
Let $f_0$ belong to $\cF_0$
and the prior $\Pi$ be defined by \eqref{pr:omk}--\eqref{cond-dir}. Let
the prior $\pi_k$ be either the Dirac mass at $k=K_n \le n/(\log n)^2$,
or the law given in \eqref{prrk}. Let $\cK_n$ be a subset of $\{
1,2,\ldots,n/\log^2{n}\}$
such that $\Pi(\cK_n| Y^n)=1+o_p(1)$.

Consider estimating a functional $\psi(f)$, with $\tilde r$ in \eqref
{func}, verifying \eqref{csimp}
and, for any $M>0$, with $A_{n,k}(M)$ defined in \eqref{def-vepan},
%
\begin{equation}
\label{rem-rand-h} \sup_{k\in\cK_n} \sup_{f\in A_{n,k}(M)}
\sqrt{n} \tilde r(f,f_0)= o_p(1),
\end{equation}
as $n\to\infty$. Additionally, suppose
%
\begin{equation}
\label{hnob} \max_{k \in\cK_n } \rn\biggl \vert\int(\tilde\psi- \tilde
\psi_{[k]}) (f_{0[k]}-f_0) \biggr\vert= o(1).
\end{equation}
Then the BvM theorem for the functional $\psi$ holds.
\end{prop}

The core condition is \eqref{hnob}, which can be seen as a \emph
{no-bias} condition. Condition \eqref{rem-rand-h} controls the remainder
term of the expansion of $\psi(f)$ around $f_0$.
Condition \eqref{csimp} is satisfied under very mild conditions: for
its first part it is enough that $\inf\cK_n$ goes to $\infty$ with
$n$. For the second part, barely more than this typically suffices,
using a simple empirical process argument; see Section~\ref{sec:proofs}.

The next theorem investigates the previous conditions under
deterministic and random priors on $k$, for the examples of functionals
\ref{exa-lin} to \ref{exa-pow}.

\begin{thmm} \label{bvm:hist} 
Suppose $f_0\in\cC^\beta$, with $\beta>0$.
Let two priors $\Pi_1, \Pi_2$ be defined by \eqref{pr:omk}--\eqref{cond-dir}
and the prior on $k$ be either the Dirac mass at $k=K_n = \lfloor
n^{1/2}(\log n)^{-2}\rfloor$
for $\Pi_1$,
or $k\sim\pi_k$ given by \eqref{prrk} for $\Pi_2$.
Then:
%
\begin{itemize}
\item Example~\ref{exa-lin}, linear functionals $\psi(f) = \int a f$,
under the prior $\Pi_1$ with deterministic $k = K_n$
\begin{itemize}
\item[$\diamond$] if $a(\cdot) \in\mathcal C^\gamma$ with $\gamma+
\beta> 1$ for some $\ga>0$, then the BvM theorem holds for the
functional $\psi(f)$;
\item[$\diamond$] if $a(\cdot) = \1_{\cdot\leq z}$ for $z\in[0, 1]$,
then BvM holds for the functional $\int\1_{\cdot\leq z}f =F(z)$, the
cumulative distribution function of $f$.
\end{itemize}
\item Examples \ref{exa-ent}--\ref{exa-sqr}--\ref{exa-pow}. For all
$\beta>1/2$, the BvM theorem holds for $\psi(f)$ for both priors $\Pi
_1$ (deterministic $k$) and $\Pi_2$ (random $k$).
\end{itemize}
\end{thmm}

Theorem~\ref{bvm:hist} is proved in Section~\ref{sec:pr:bvm:hist}. From
this proof, it may be noted that different choices of $K_n$ in some
range lead to similar results for some examples.
For instance, if $\psi(f) = \int\psi f$ and $\psi\in\mathcal
C^\gamma$, choosing $K_n = \lfloor n/(\log n)^2 \rfloor$ implies that
the BvM holds for all $\gamma+ \beta>1/2$. 

Obtaining BvM in the case of a prior with random $k$ in Example~\ref
{exa-lin} is case-dependent. The answer lies in the respective
approximation properties of both $f_0$ and $\tilde\psi_{f_0}$ through
the prior (note that a random $k$ prior typically adapts to the
regularity of $f_0$), and the no-bias condition \eqref{hnob} may not be
satisfied if $\inf\mathcal K_n$ is not large enough.

We present below a counterexample where BvM is proved to fail for a
large class of true densities $f_0$ when a prior with random $k$ is chosen.

\subsection{A semiparametric curse of adaptation: A counterexample for
BvM under random number of bins histogram priors} \label{counter}
Consider a $\cC^1$, strictly increasing true function $f_0$, say
%
\begin{equation}
\label{ce-fz} f_0'\ge\rho>0 \qquad\mbox{on } [0,1].
\end{equation}
The following reasoning can be extended to any approximately monotone
smooth function on $[0,1]$.
Consider estimation of the linear functional $\psi(f)=\int\psi f$. The
BvM theorem is not satisfied if the bias term $\rn(\hat\psi- \hat
\psi
_k)$ is predominant for all $k$'s which are asymptotically given mass
under the posterior. This will happen if for all such $k$'s,
\[
-b_{n,k}=\rn\int\psi(f_0 - f_{0[k]}) = \rn\int(
\psi- \psi_{[k]}) (f_0 - f_{0[k]}) \gg1,
\]
as $n\to\infty$.
To simplify the presentation, we restrict ourselves to the case of
dyadic random histograms; in other words, the prior on $k$ only puts
mass on values of $k= 2^p$, $p \geq0$. Then define $\psi(x)$ as, for
$\al>0$,
%
\begin{equation}
\label{ce-psi} \psi(x) = \sum_{l\ge0} \sum
_{j=0}^{2^l - 1} 2^{-l({1}/2 + \al)} \psi
_{lj}^H(x),
\end{equation}
where $\psi_{lj}^H(x)=2^{l/2}\psi_{00}(2^lx-j)$ and $\psi_{00}(x)=-\1
_{[0,1/2]}(x)+\1_{(1/2,1]}(x)$ is the mother wavelet of the Haar basis
(we omit the scaling function $1$ in the definition of $\psi$).

\begin{prop} \label{ce}
Let $f_0$ be any function as in \eqref{ce-fz} and $\alpha, \psi$ as in
\eqref{ce-psi}. Let the prior be as in Theorem~\ref{bvm:hist}. Then
there exists $k_1 >0$ such that
\[
\Pi \bigl( k < k_1 (n/\log n)^{1/3} | Y^n
\bigr) = 1 + o_P(1)
\]
and for all $ p\in\N$ such that $ 2^p:=K < k_1 (n/\log n)^{1/3}$,
the conditional
posterior distribution of $\rn(\psi(f)-\hat\psi-b_{n,k})/\sqrt
{V_k}| k=K$ converges in distribution to $\cN(0,1)$, in
$P_0^n$-probability, with
\[
b_{n,K} \leqa-\sqrt{n} K^{-\al-1}.
\]
In particular, the BvM property does not hold if $\al<1/2$.
\end{prop}

\begin{rem}
For the considered $f_0$, it can be checked that the posterior even
concentrates on values of $k$ such that
$k=k_n \asymp(n/\log n)^{1/3}$.
\end{rem}

As soon as the regularities of the functional $\psi(f)$ to be estimated
and of the true function
$f_0$ are fairly different, taking an \textit{adaptive} prior (with
respect to $f$) can have disastrous effects
with a nonnegligible bias appearing in the centering of the posterior
distribution. As in the
counterexample in Rivoirard and Rousseau~\cite{rivoirard:rousseau:09},
the BvM is ruled out
because the posterior distribution concentrates on values of $k$ that
are too small and for which the bias $b_{n,k}$ is not negligible. Note
that for each of these functionals the BvM is violated for a large
class of true densities $f_0$. Some related phenomena in terms of rates
are discussed in Knapik et al. \cite{ksvv12} for linear functionals
and adaptive
priors in white noise inverse problems.

Let us sketch the proof of Proposition~\ref{ce}. It is not difficult to
show that (see the Supplement), since $f_0\in\cC^1$,
the posterior
concentrates on the set
$\{ f: \|f - f_0\|_1 \leq M (n/\log n)^{-1/3}, k \leq k_1 (n/\log
n)^{1/3} \}$, for some positive $M$ and $k_1$.
Since Haar wavelets are special cases of (dyadic) histograms, for any
$K\ge1$ the best approximation of $\psi$ within $\cH_K$ is
\[
\psi_{[K]}(x)=\sum_{l=0}^p
\sum_{j=0}^{2^l - 1} 2^{-l({1}/2 + \al)}
\psi_{lj}^H(x).
\]
The semiparametric bias $-b_{n,K}$ is equal to
$\rn\int_0^1 (f_0 - f_{0,[K]})(\psi- \psi_{[K]})=\rn\int_0^1
f_0(\psi
- \psi_{[K]})$, which can be written,
for any $K \ge1$,
\begin{eqnarray*}
-b_{n,K} & = &\sqrt{n}\sum_{l>p} \sum
_{j=0}^{2^l - 1} 2^{-l({1}/2
+ \al)} \int
_0^1 f_0(x) \psi_{lj}^H(x)
\,dx
\\
& =& \sqrt{n} \sum_{l>p} \sum
_{j=0}^{2^l - 1} 2^{-l \al} \int
_{2^{-l}j}^{2^{-l}(j+1/2)} \bigl(f_0
\bigl(x+2^{-l}/2 \bigr)-f_0(x) \bigr)\,dx
\\
& \geqa&\sqrt{n} \sum_{l>p} 2^{-l\al}
2^{l}2^{-2l} \geqa\sqrt{n} K^{-\al-1}.
\end{eqnarray*}
Since $\Pi(k\le n^{1/3}| Y^n) = 1+o_p(1)$, we have that
$\inf_{k \leq n^{1/3}} -b_{n,k} \rightarrow+\infty$ for all $\al
<1/2$. Also, the sequence of real numbers $\{V_k\}_{k\ge1}$ stays
bounded, while the supremum $\sup_{1\le k\le n^{1/3}} |\mathbb
{G}_n(\tilde\psi-\tilde\psi_{[k]})|$ is bounded by a constant times
$(\log n)^{1/2}$ in probability, by a standard empirical process
argument. This implies that
\[
E^\Pi \bigl[ e^{ t\rn(\psi(f)-\hat\psi) } | Y^n, B_n
\bigr] = \bigl( 1 + o(1) \bigr) \sum_{k\in\cK_n}
e^{ t^2 V_k/2 + t\rn(\hat\psi
-\hat\psi
_k)} \Pi \bigl[k| Y^n \bigr] = o_p(1),
\]
so that the posterior distribution is not asymptotically equivalent to\break
$\cN(0,\|\tilde\psi\|_L^2)$, and there exists $M_n $ going to infinity
such that
\[
\Pi \bigl[\rn|\psi(f)-\hat\psi| > M_n | Y^n \bigr] = 1 +
o_p(1).
\]

\subsection{Gaussian process priors} \label{subsec:dens:gaussian}
We now investigate the implications of Theorem~\ref{thmm-fundensity} in
the case of Gaussian process priors for the density $f$. Consider as a
prior on $f$ the distribution on densities generated by
%
\begin{equation}
\label{def-gp} f(x) = \frac{ e^{W(x)} }{ \int_0^1 e^{W(x)} \,dx },
\end{equation}
where $W$ is a zero-mean Gaussian process indexed by $[0,1]$ with
continuous sample paths.
The process $W$ can also be viewed as a random element in the Banach
space $\mb$ of continuous functions on $[0,1]$ equipped with the
sup-norm $\|\cdot\|_\infty$; see \cite{rkhs} for precise definitions.
We refer to \cite{rkhs,vvvz} and \cite{ic08} for basic
definitions on Gaussian priors and some convergence properties,
respectively. Let $K(x,y)= E[W(x)W(y)]$ denote the covariance kernel of
the process and let $(\mh,\| \cdot\|_{\mathbb H})$ denote the reproducing
kernel Hilbert space of $W$.

\begin{example}[(Brownian motion released at $0$)] Consider the
distribution induced by
\[
W(x) = N + B_x, \qquad x\in[0,1],
\]
where $B_x$ is standard Brownian motion and $N$ is an independent $\cN
(0,1)$ variable.
We use it as a prior on $w$. It can be seen (see \cite{vvvz}) as a
random element in the
Banach space $\mb= (\cC^0,\|\cdot\|_\infty)$ and its RKHS is
\[
\mh^B = \biggl\{ c + \int_0^{\cdot}
g(u)\,du, c\in\RR, g\in L^2[0,1] \biggr\},
\]
a Hilbert space with norm 
given by $ \|c + \int_0^{\cdot} g(u)\,du\|_{\mh^B}^2 = c^2 + \int_0^1
g(u)^2 \,du$.
\end{example}

\begin{example}[(Riemann--Liouville-type processes)] Consider the
distribution induced by,
for $\al>0$ and $x\in[0,1]$,
\[
W^{\al}(x)= \sum_{k=0}^{\lfloor\al\rfloor+1}
Z_k x^k + \int_0^x
(x-s)^{\al-1/2}\,dB_s,
\]
where $Z_i$s are independent standard normal variables and $B$ is an independent
Brownian motion. The RKHS $\mh^\al$ of $W^\al$ can be obtained
explicitly from the one of Brownian motion, and is nothing but a
Sobolev space of order $\alpha+1$; see \cite{vvvz}, Theorem~4.1.
\end{example}

The concentration function of the Gaussian process in $\mb$ at $\eta_0
= \log f_0$ is defined for any $\veps>0$ by (see \cite{rkhs})
\[
\vphi_{\eta_0}(\veps) = -\log\Pi\bigl(\|W\|_{\infty} \le\veps\bigr) +
\frac{1}{2} \inf_{h\in\mh: \|h-\eta_0\|_{\mb}<\veps} \|h\| _{\mh}^2.
\]
In van der Vaart and van Zanten \cite{vvvz}, it is shown that the
posterior contraction rate for
such a prior is closely connected to a solution $\veps_n$ of
%
\begin{equation}
\label{eq-vphi} \vphi_{\eta_0}(\veps_n) \le n
\veps_n^2,\qquad \eta_0=\log f_0.
\end{equation}

\begin{prop} \label{thmm-gp}
Suppose $f_0$ verifies $c_0\leq f_0 \leq C_0$ on $[0,1]$, for some
positive $c_0, C_0$.
Let the prior $\Pi$ on $f$ be induced via a Gaussian process $W$ as in
\eqref{def-gp} and let $\mh$ denote its RKHS. Let $\veps_n\to0$ verify
\eqref{eq-vphi}.
Consider estimating a functional $\psi(f)$, with $\tilde r$ in \eqref
{func} verifying
\[
\sup_{f\in A_n} \tilde r(f,f_0) = o(1/\sqrt{n}),
\]
for $A_n$ such that $\Pi(A_n| Y^n)=1+o_p(1)$ and $A_n\subset\{f:
h(f,f_0)\le\veps_n\}$.
Suppose that $\tilde{\psi}_{f_0}$ is continuous and that there exists
a sequence
$\psi_n \in\mathbb H$ and $\zeta_n \to0$, such that
%
\begin{eqnarray}
\label{approx} \| \psi_n - \tilde{\psi}_{f_0}
\|_\infty&\le& \zeta_n\quad\mbox{and}\quad \|\psi_n
\|_{\mh} \le\rn \zeta_n, 
\\
\rn\veps_n\zeta_n &\to& 0. \label{no-bias}
\end{eqnarray}
Then, for $\hat{\psi}$ any linear efficient estimator of $\psi(f)$, in
$P_{0}^n$-probability, the posterior distribution of $\rn(\psi
(f)-\hat{\psi}
) $ converges to a Gaussian\vspace*{1pt} distribution with mean $0$ and variance $\|
\tilde{\psi}_{f_0}\|_L^2$ and the BvM theorem holds.
\end{prop}

The proof is presented in Section~3.2 of Castillo and Rousseau \cite
{castillo:rousseau:suppl}. 
We now investigate conditions \eqref{approx}--\eqref{no-bias} for
examples of Gaussian priors.

\begin{thmm} \label{thmgp}
Suppose that $\eta_0=\log f_0$ belongs to $\cC^\beta$, for some
$\beta>0$.
Let $\Pi_\alpha$ be the priors defined from a Gaussian process $W$ via
\eqref{def-gp}. For $\Pi_1$, we take $W$ to be
Brownian motion (released at $0$) and for $\Pi_2$ we take $W=W^{\al}$,
a Riemann--Liouville-type process of parameter $\al>0$. 
%
\begin{itemize}
\item Example~\ref{exa-lin}, linear functionals $\psi(f) = \int a f$
\begin{itemize}
\item[$\diamond$] if $a(\cdot) \in\mh^B$, then the BvM theorem holds
for the functional $\psi(f)$ and prior $\Pi_1$. The same holds if
$a(\cdot) \in\mh^\al$ for prior $\Pi_2$;
\item[$\diamond$] if $a(\cdot)\in\cC^\mu$, $\mu>0$, the BvM property
holds for prior $\Pi_2$ if
\[
\alpha\wedge\beta> \tfrac{1}2 + (\al-\mu)\vee0.
\]
\end{itemize}
\item Examples \ref{exa-sqr}--\ref{exa-pow}. Under the same condition as
for the linear functional with $\mu=\beta$, the BvM theorem holds for
$\Pi_2$.
\end{itemize}
\end{thmm}

An immediate illustration of Theorem~\ref{thmgp} is as follows.
Consider prior $\Pi_1$ built from Brownian motion. Then for all linear
functionals
\[
\psi(f) = \int_0^1 x^{r} f(x) \,dx,\qquad
r>\frac{1}{2},
\]
the BvM theorem holds. Indeed, $x\to x^r, r>1/2$ belongs to $\mh^B$.



To prove Theorem~\ref{thmgp}, one applies Proposition~\ref{thmm-gp}: it is enough to compute bounds for $\veps_n$ and $\zeta
_n$. This follows from the results on the concentration function for
Riemann--Louville-type processes obtained in Theorem~4 in \cite{ic08}.
For linear functionals $\psi(f)=\int af$ and $a\in\cC^\mu$, one can
take $\veps_n=n^{-\alpha\wedge\beta/(2\al+1)}$ and $\zeta
_n=n^{-\mu
/(2\al+1)}$, up to some logarithmic factors.
So \eqref{no-bias} holds if $\alpha\wedge\beta>\frac{1}2 + (\al
-\mu
)\vee0$.

The square-root functional is similar to a linear functional with $\mu
=\beta$, since the remainder term in the expansion of the functional is
of the order of the Hellinger distance. Indeed,
since $f_0$ is bounded away from $0$ and $\infty$, the fact that
$w_0\in\cC^\beta$ implies that $f_0\in\cC^{\beta}$ and $\sqrt
{f_0}\in
\cC^{\beta}$.
For power functionals, the remainder term $r(f,f_0)$ is more
complicated but is easily bounded
by a linear combination of terms of the type
\[
\int(f-f_0)^{2+r} f_0^{q-2-r} \le
\|f_0\|_\infty^{q-r-2} \|f-f_0\|
_{\infty}^r \int(f-f_0)^2.
\]
Using Proposition~1 in Castillo and Rousseau \cite
{castillo:rousseau:suppl}, one obtains that,
under the posterior distribution, $\|f-f_0\|_{\infty}\leqa1$ and $\|
f-f_0\|_2\leqa\veps_n$. So, $\rn r(f,f_0)=o(1)$ holds
if $\rn\veps_n^2=o(1)$, which is the case since $\al\wedge\beta>1/2$.

\section{Application to the nonlinear autoregressive model} \label
{sec:autoreg}
Consider an autoregressive model in which one observes $Y_1, \ldots, Y_n$
given by
%
\begin{equation}
\label{autoR:model} Y_{i+1} = f(Y_i) + \epsilon_i,\qquad
\epsilon_i \sim\mathcal N(0, 1) \qquad\mbox{ i.i.d.},
\end{equation}
where $\|f\|_\infty\leq L$ for a fixed given positive constant $L$ and
$f$ belongs to a H\"older space $\mathcal C^\beta$, $\beta>0$.
This example has been in particular studied by \cite{gvni} and it is
known that $(Y_i, i=1,\ldots, n) $ is an homogeneous Markov chain and
that under these assumptions, for all $f$,
there exists a unique stationary distribution $Q_f$ with density $q_f$
with respect to Lebesgue measure. The transition density is
$p_f(y|x)=\phi(y-f(x))$. Denoting
$r(y) = ( \phi( y - L) + \phi( y+L))/2$, the transition density satisfies
$p_f(y|x) \asymp r(y)$ for all $x,y \in\R$. Following \cite{gvni},
define the norms, for any $s\ge2$,
\[
\| f - f_0 \|_{s,r} = \biggl(\int_{\R}
\bigl| f(x) - f_0(x) \bigr|^s r(x) \,dx \biggr)^{1/s}.
\]
%

As in \cite{gvni}, we consider a prior $\Pi$ on $f$ based on piecewise
constant functions. Let us set $ a_n = b \sqrt{\log n}$,
where $b>0$ and consider functions $f$ of the form
\[
f(x) := f_{\omega,k} (x) = \sum_{j=0}^{k-1}
\omega_j \1_{I_j}(x), \qquad I_j= a_n
\bigl( \bigl[j/k, (j+1)/k \bigr] - 1/2 \bigr).
\]
A prior on $k$ and on $\omega=(\omega_0,\ldots,\omega_{k-1})$ is then
specified as follows. First, draw $k\sim\pi_k$, for $\pi_k$ a law on
the integers. Given $k$, the law $\omega| k$ is supposed to have a
Lebesgue density $\pi_{\omega| k}$ with support $[-M,M]^k$ for some
$M>0$. Assume further that these laws satisfy, for
$ 0 <c_2 \leq c_1 < \infty$
and $C_1, C_2>0$,
%
\begin{eqnarray}
\label{def:prior:AR} %
e^{-c_1 K \log K }&\leq&\pi_k[k > K] \leq
e^{-c_2 K\log K}\qquad \mbox{for large } K,
\nonumber
\\[-8pt]
\\[-8pt]
\nonumber
e^{-C_2 k \log k} &\leqa& \pi_{\omega|k} (\omega) \le C_1\qquad
\forall\omega\in[-M,M]^k. %
\end{eqnarray}
We consider the squared-weighted-$L_2$ norm functional
$\psi(f) = \int_{\R} f^2 (y) q_f(y) \,dy $. As before, define
\[
k_n(\beta) = \bigl\lfloor(n/\log n)^{1/(2\beta+ 1)} \bigr\rfloor,\qquad
\veps_n(\beta) = (n/\log n)^{-\beta/(2\beta+1)}.
\]
For all bounded $f_0$ and all $k >0$, define
\[
\tilde\omega_{[k]}^0 = \bigl(\tilde\omega_1^0,
\ldots, \tilde\omega_k^0 \bigr), \qquad \tilde
\omega_j^0 = \frac{ \int_{I_j} f_0(x) q_{f_0}(x) \,dx }{\int_{I_j}
q_{f_0}(x) \,dx };
\]
these are the weights of the projection of $f_0$ on the weighted space
$L^2(q_{f_0})$.
We then have the following sufficient condition for the BvM to be valid.

\begin{thmm}\label{bvm:AR}
Consider the autoregressive model \eqref{autoR:model} and the prior
\eqref{def:prior:AR}. Assume that $f_0 \in\mathcal C^\beta$, with
$\beta>1/2$ and $\|f_0\|_\infty< L$, and assume that $\pi_{\omega|k}$
satisfies for all $t>0$ and all $M_0>0$
%
\begin{equation}
\label{cond:prior:AR} \sup_{\| \omega- \tilde\omega_{[k]}^0\|
_{2,r} \leq M_0 \veps
_n(\beta
)} \biggl\vert\frac{ \pi_{\omega|k}(\omega- t\tilde\omega
^0_{[k]}/\sqrt
{n})}{ \pi_{\omega|k}(\omega)} - 1 \biggr\vert=
o(1).
\end{equation}
Then the posterior distribution of $\sqrt{n}(\psi(f) -\hat\psi) $ is
asymptotically Gaussian with mean $0$ and variance $V_0$, where
\[
\hat\psi= \psi(f_0) + \frac{2}{n}\sum
_{i=1}^n \epsilon_i
f_0(Y_{i-1})+ o_p \bigl(n^{-1/2}
\bigr),\qquad V_0 = 4\|f_0\|_{2,q_{f_0}}^2
\]
and the BvM is valid under the distribution associated to $f_0$ and any
initial distribution $\nu$ on $\R$.
\end{thmm}

Theorem~\ref{bvm:AR} is proved in Section~4 of Castillo and Rousseau
\cite{castillo:rousseau:suppl}.
The conditions on the prior \eqref{def:prior:AR} and \eqref
{cond:prior:AR} are satisfied in particular when $k \sim\mathcal
P(\lambda)$ and when given $k$, the law $\omega| k$ is the independent
product of $k$ laws $\mathcal U(-M,M)$.
Theorem~\ref{bvm:AR} is an application of the general Theorem~\ref
{lem:laplace}, with $A_n = \{ f_{\omega, k}; k \leq k_1 k_n(\beta); \|
\omega- \omega_{[k]}^0\|_{2,r}\leq M_0\varepsilon_n(\beta)\}$ and
Assumption~\ref{assa} implied by $\beta>1/2$. Condition \eqref
{cond:prior:AR} is used to prove condition \eqref{laplace}.

\section{Proofs} \label{sec:proofs}

\subsection{Proof of Theorem \texorpdfstring{\protect\ref{lem:laplace}}{2.1}} \label{sec:pr:lem:laplace}
Let the set $A_n$ be as in Assumption~\ref{assa}.
Set
\[
I_n := E \bigl[ e^{t \sqrt{n} (\psi(\eta) - \psi(\eta_0)) } | Y^n, A_n
\bigr].
\]
For the sake of conciseness, we prove the result in the case where
$\psi_0^{(2)}\neq0$ since the other case is a simpler version of it. Using
the LAN expansion \eqref{LAN} together with the expansion \eqref
{smooth:psi} of the functional $\psi$, one can write
\begin{eqnarray*}
I_n &=& \frac{ \int_{A_n} e^{\sqrt{n}t ( \langle\psi_0^{(1)}, \eta
-\eta
_0\rangle_L +({1}/{2}) \langle\psi_0^{(2)}(\eta-\eta_0), \eta
-\eta
_0\rangle
_L )
+\ell_n(\eta)-\ell_n(\eta_0)+ t\sqrt{n}r(\eta,\eta_0) } \,d\Pi
(\eta)
}{ \int_{A_n} e^{ { -n \Vert \eta-\eta_0\Vert^2_L}/{ 2 } + \sqrt
{n}W_n(\eta- \eta_0)+ R_n(\eta,\eta_0) } \,d\Pi(\eta)}.
\end{eqnarray*}
Consider, for any real number $t$, as defined in \eqref{etat2},
\[
\eta_t = \eta- \frac{ t \psi_0^{(1)}}{ \sqrt{n}} -\frac{
t}{2\sqrt
{n}}
\psi_0^{(2)} (\eta- \eta_0) -
\frac{t \psi_0^{(2)} w_n }{2n}.
\]
Then using \eqref{zetaAn}--\eqref{zetapsi} in Assumption~\ref{assa},
on $A_n$,
\begin{eqnarray*}
&& \ell_n(\eta_t) - \ell_n(
\eta_0) - \bigl(\ell_n(\eta) -\ell_n(\eta
_0) \bigr)
\\
&&\qquad= - \frac{ n}{2} \bigl[ \Vert\eta_t -\eta_0\Vert^2_{ L} - \Vert\eta-\eta_0\Vert^2_{ L} \bigr] + \sqrt{n}\langle w_n ,
\eta_t- \eta\rangle_L + R_n(
\eta_t, \eta_0) \\
&&\qquad\quad{}- R_n( \eta,
\eta_0) + o_P(1)
\\
&&\qquad= - t \bigl\langle w_n , \psi_0^{(1)}+
\psi_0^{(2)}w_n /(2 \sqrt{n}) \bigr
\rangle_L - \frac{
t^2}{2} \biggl\Vert\psi_0^{(1)}+
\frac{\psi_0^{(2)}w_n }{2\sqrt
{n} }\biggr \Vert_L^2\\
&&\qquad\quad{} + \sqrt{n}t \bigl\langle
\psi_0^{(1)}, \eta- \eta_0\bigr
\rangle_L
\\
&&\qquad\quad{} + \frac{t \sqrt{n}}{2}\bigl\langle\psi_0^{(2)}(\eta-
\eta_0) , \eta- \eta_0\bigr\rangle_L +
R_n( \eta_t, \eta_0) - R_n(
\eta, \eta_0) + o_P(1). 
\end{eqnarray*}
One deduces that on $A_n$, from \eqref{Rn} in Assumption~\ref{assa},
\begin{eqnarray*}
&&\sqrt{n}t \biggl( \bigl\langle\psi_0^{(1)}, \eta-
\eta_0\bigr\rangle_L +\frac{1}{2} \bigl\langle
\psi_0^{(2)}(\eta-\eta_0), \eta-
\eta_0\bigr\rangle_L \biggr)+\ell_n(\eta) -
\ell_n(\eta_0) \\
&&\quad{}+ \sqrt{n}t r (\eta, \eta_0)
\\
&&\qquad= \ell_n(\eta_t) - \ell_n(
\eta_0) + t \bigl\langle w_n , \psi_0^{(1)}+
\psi_0^{(2)}w_n /(2 \sqrt{n}) \bigr
\rangle_L \\
&&\qquad\quad{}+\frac{ t^2}{2} \biggl\Vert\psi_0^{(1)}+
\frac
{\psi_0^{(2)}
w_n }{2\sqrt{n} }\biggr \Vert_L^2 +o_P(1).
\end{eqnarray*}
We can then rewrite $I_n$ as
\[
I_n = e^{o_P(1) + \frac{t^2 }{ 2 }\Vert \psi_0^{(1)} +\frac{\psi
_0^{(2)} w_n}{2 \sqrt{n}} \Vert_L^2 + t \langle w_n , \psi_0^{(1)}
+ \frac{\psi_0^{(2)} w_n}{ 2 \sqrt{n}} \rangle_L }
\frac{ \int_{A_n} e^{\ell_n(\eta_t) - \ell_n(\eta_0) } \,d\Pi(\eta
) }{
\int_{A_n} e^{\ell_n(\eta) - \ell_n(\eta_0) } \,d\Pi(\eta)},
\]
and Theorem~\ref{lem:laplace} is proved using condition \eqref
{big:cond}, together with the fact that (see Section~1 of Castillo and
Rousseau \cite
{castillo:rousseau:suppl}), convergence of Laplace transforms for all
$t$ in probability implies convergence in distribution in probability.

\subsection{Proof of Theorem \texorpdfstring{\protect\ref{thmm-fundensity}}{4.1}} \label
{sec:lem:dens}
One can define
$\psi_0^{(1)}= \tilde{\psi}_{f_0}+ c$ for any constant $c$, since
the inner product
associated to the LAN norm corresponds to re-centered quantities. In
particular, for all $\eta= \log f$
\[
\bigl\langle(\tilde{\psi}_{f_0}+ c), \eta- \eta_0 \bigr
\rangle_L = \int (\tilde{\psi}_{f_0}- P_{f_0}
\tilde{\psi}_{f_0}) (\eta- \eta_0) f_0,\qquad \|
\tilde{\psi}_{f_0}+ c \|_L = \|\tilde{
\psi}_{f_0}\|_L.
\]
To check Assumption~\ref{assa}, let us write
%
\begin{equation}
\label{psiu:dens} \psi_0^{(1)}=\tilde{\psi}_{f_0}+
\frac{\sqrt
{n}}{t} \log \biggl( \int_0^1e^{\eta
-
({t}/{\rn})\tilde{\psi}_{f_0}}(x)\,dx
\biggr),
\end{equation}
which depends on $\eta$ but is of the form $\tilde{\psi}_{f_0}+ c$
(see also Remark~\ref{com:psiu}), and we study
$\sqrt{n} t r(\eta, \eta_0) + R_n(\eta, \eta_0) - R_n(\eta_t,
\eta_0)$
using Rivoirard and Rousseau's \cite{rivoirard:rousseau:09} calculations pages 1504--1505.
Indeed, writing $h = \sqrt{n}( \eta- \eta_0)$ we
have
\[
R_n(\eta, \eta_0) - R_n(
\eta_t, \eta_0) = t \langle h, \tilde{
\psi}_{f_0}\rangle_L - \frac
{ t^2 }{ 2 }\|\tilde{
\psi}_{f_0}\|_L^2 + n \log F
\bigl[e^{-t \tilde{\psi}_{f_0}/\sqrt{n}} \bigr]
\]
and expanding the last term as in page 1506 of \cite
{rivoirard:rousseau:09} we obtain that
\begin{eqnarray*}
n \log F \bigl[e^{-t \tilde{\psi}_{f_0}/\sqrt{n}} \bigr]&=& n \log \biggl( 1 - \frac{t }{n}
\langle h, \tilde{\psi}_{f_0}\rangle_L -
\frac{t }{ \sqrt{n} } \mathcal B(f,f_0)+ \frac{ t^2}{2n} \| \tilde{
\psi}_{f_0}\|_L^2
\\
&&{} + \frac{ t^2}{2n}(F-F_0) \bigl(\tilde{\psi}_{f_0}^2
\bigr)+ O \bigl(n^{-3/2} \bigr) \biggr)
\\
&=& - t \langle h, \tilde{\psi}_{f_0}\rangle_L- t
\sqrt{n}\mathcal B(f,f_0) + \frac{
t^2}{2} \| \tilde{
\psi}_{f_0}\|_L^2 \\
&&{}+ O \bigl(\|f
-f_0\|_1 + n^{-1/2} \bigr)
\\
&=&- t \langle h, \tilde{\psi}_{f_0}\rangle_L- t \sqrt{n}
\mathcal B(f,f_0) + \frac{
t^2}{2} \| \tilde{
\psi}_{f_0}\|_L^2 +o(1)
\end{eqnarray*}
since $|(F-F_0)( \tilde{\psi}_{f_0}^2) \leq\|\tilde{\psi}_{f_0}\|
_\infty^2 \|f - f_0\|_1
\lesssim\veps_n$ on $A_n$. Finally, this implies that
$ \sqrt{n} t r(\eta, \eta_0) + R_n(\eta, \eta_0) - R_n(\eta_t,
\eta_0)
= o(1)$ uniformly over $A_n$ and Assumption~\ref{assa} is satisfied.

\subsection{Proof of Theorem \texorpdfstring{\protect\ref{bvm:hist}}{4.2}} \label{sec:pr:bvm:hist}
The first part of the proof consists in establishing that the posterior
distribution on random histograms concentrates (a) given the number of
bins $k$, around the projection $f_{0,[k]}$ of $f_0$, and (b) globally
around $f_0$ in terms of the Hellinger distance.

More precisely, (a) there exist $c,M>0$ such that
%
\begin{equation}
\label{misspecified:1} P_0 \biggl[ \exists k \leq\frac{n}{\log n}; \Pi
\bigl[f\notin A_{n,k}(M) | Y^n, k \bigr] > e^{ - c k \log n}
\biggr] = o(1).
\end{equation}
(b) Suppose $f_0\in\cC^\beta$ with $0<\beta\leq1$.
If $k_n(\beta) = (n/\log n)^{1/(2\beta+1)}$ and $\veps_n(\beta) =
k_n(\beta)^{-\beta}$, then for $k_1,M$ large enough,
%
\begin{equation}
\label{conc:rand} \Pi \bigl[ h(f_0, f) \le M\veps_n(
\beta); k\leq k_1 k_n(\beta) | Y^n \bigr] = 1
+ o_{p}(1).
\end{equation}
Both results are new. As (a)--(b) are an intermediate step and concern
rates rather than BvM per se, their proofs are given in Castillo and
Rousseau \cite
{castillo:rousseau:suppl}.

We now prove that the BvM holds if there exists $\mathcal K_n$ such
that $\Pi(\mathcal K_n|Y^n) = 1 + o_p(1)$, and for which
%
\begin{equation}
\label{bias1} \sup_{k\in\cK_n} \sqrt{n} |\hat\psi-\hat
\psi_k| = o_p(1), \qquad\sup_{k\in\cK_n}
|V_k - V| = o_p(1),
\end{equation}
for all $\psi(f) $ satisfying \eqref{func} with
%
\begin{equation}
\label{rem-rand} \sup_{k \in\mathcal K_n} \sup_{f \in A_{n,k}(M)}
\tilde r (f; f_0) = o_p(1).
\end{equation}
Consider first the deterministic $k=K_n$ number of bins case.
The study of the posterior distribution of $\sqrt{n}( \psi(f) - \hat
\psi)$ is based on a slight modification of the proof of Theorem~\ref
{thmm-fundensity}. Instead of taking the true $f_0$ as basis point for
the LAN expansion, we take instead $f_{0,[k]}$. This enables to write
the main terms in the LAN expansion completely within $\cH_k$.

Let us define $\bar\psi_{(k)}:= \psi_{[k]} - \int\psi_{[k]} f_{0,[k]}
= \tilde\psi_{[k]} -
\int\tilde\psi_{[k]} f_{0,[k]}$ and $\hat\psi_k= \psi( f_{0,[k]}
) +
\frac{1}{\rn}W_n(\bar\psi_{(k)})$.
With the same notation as in Section~\ref{sec:density}, where
indexation by $k$ means that $f_0$ is replaced by $f_{0,[k]}$ [in $\|
\cdot\|_{L,k}, R_{n,k}$, etc., where one can note that for $g\in\cH
_k$, one has $W_{n,k}(g)=W_n(g)$],
\begin{eqnarray*}
&& t\rn \bigl(\psi(f)-\hat\psi_k \bigr) + \ell_n(f)-
\ell_n(f_{0,[k]})
\\
&&\qquad= -\frac{n}2\biggl\|\log\frac{f}{f_{0,[k]}} - \frac{t}{\rn}
\bar\psi_{(k)}\biggr\|_{L,k}^2 \\
&&\qquad\quad{}+ \rn W_n
\biggl(\log\frac{f}{f_{0,[k]}} - \frac{t}{\rn}\bar\psi _{(k)}
\biggr)
\\
&&\qquad\quad{} + \frac{t^2}{2}\|\bar\psi_{(k)}\|_{L,k}^2
+ t\rn\cB_{n,k}+R_{n,k}(f,f_{0,[k]}).
\end{eqnarray*}
Let us set
$f_{t,k}= f e^{-{t\bar\psi_{(k)}}/{\rn}} / F(e^{-{t\bar
\psi_{(k)}}/{\rn}})$.
Then, using the same arguments as in Section~\ref{sec:density},
together with \eqref{misspecified:1} and the fact that $\int\bar\psi
_{(k)} f_{0,[k]}=0$,
\begin{eqnarray*}
&&t\rn \bigl(\psi(f)-\hat\psi_k \bigr) + \ell_n(f)-
\ell_n(f_{0,[k]}) \\
&&\qquad= \frac{t^2}{2}\| \bar
\psi_{(k)} \|_{L,k}^2 + \ell_n(f_{t,k})-
\ell_n(f_{0,[k]}) +o(1),
\end{eqnarray*}
so that choosing $A_{n,k} = \{ \omega\in\mathcal S_k; \| f_{\omega,k}
- f_{0[k]}\|_1 \leq M \sqrt{k \log n/n} \}$, we have
\begin{eqnarray*}
&&E^\Pi \bigl[ e^{ t\rn(\psi(f)-\hat\psi_k) } | Y^n, A_{n,k}
\bigr]\\
&&\qquad = e^{({t^2}/{2})\| \bar\psi_{(k)} \|_{L,k}^2 + o(1)} \times\frac{ \int_{A_{n,k}} e^{\ell_n ( f_{t,k}) - \ell_n( f_{0[k]})
}\,d\Pi
_k( f) }{ \int_{A_{n,k}} e^{\ell_n ( f) - \ell_n( f_{0[k]}) }\,d\Pi
_k( f)},
\end{eqnarray*}
uniformly over $k = o(n/\log n)$.
Within each model $\cH_k$, since $f = f_{\omega, k}$, we can express
$f_{t,k}=k\sum_{j=1}^k \zeta_j \1_{I_j}$, with
%
\begin{equation}
\label{poids} \zeta_j = \frac{ \omega_j \ga_j^{-1}}{
\sum_{j=1}^k \omega_j \ga_j^{-1} },
\end{equation}
where we have set, for $1\le j \le k$,
$\ga_j = e^{ t\bar\psi_{j} /\rn}$, and $\bar\psi_{j} := k\int_{I_j}
\bar\psi_{(k)}$.
Denote $S_{\ga^{-1}}(\omega)=\sum_{j=1}^k \omega_j\ga_j^{-1}$. Note
that \eqref{poids} implies
$ S_{\ga^{-1}}(\omega)=S_{\ga}(\zeta)^{-1}$.
So,
\[
\frac{ \Pi_k( \omega) }{ \Pi_k( \zeta)} = \prod_{j=1}^k
e^{ t
(\alpha
_{j,k} - 1) \bar\psi_j /\sqrt{n} }S_\ga(\zeta)^{-\sum_{j=1}^k
(\alpha
_{j,k} -1)}.
\]
Let $\Delta$ be the Jacobian of the change of variable computed in
Lemma~5 of the supplemental article
(Castillo and Rousseau \cite{castillo:rousseau:suppl}).
Over the set $A_{n,k}$, it holds
\begin{eqnarray*}
&&d\Pi_k( \omega)\\
&&\qquad= \prod_{j=1}^k
e^{ t (\alpha_{j,k} - 1) \bar\psi_j
/\sqrt{n} }S_\ga(\zeta)^{- \sum_{j=1}^k (\alpha_{j,k} -1)} \Delta (\zeta) \,d
\Pi_k(\zeta)
\\
&&\qquad= S_\ga(\zeta)^{- \sum_{j=1}^k \alpha_{j,k} }e^{ t \sum_{j=1}^k
\alpha
_{j,k} \bar\psi_j /\sqrt{n} }
\,d \Pi_k(\zeta)
\\
&&\qquad= e^{ t \sum_{j=1}^k \alpha_{j,k} \bar\psi_j /\sqrt{n} } \biggl( 1 - \frac{t}{\rn} \int_0^1
\bar\psi_{(k)}(f-f_0) + O \bigl(n^{-1} \bigr)
\biggr)^{
\sum
_{j=1}^k \alpha_{j,k} } 
\,d\Pi_k(\zeta),
\end{eqnarray*}
where we have used that
\[
S_{\ga^{-1}}(\omega) = \int_0^1
e^{-t\bar\psi_{(k)}/\rn} f = 1 - \frac
{t}{\rn} \int_0^1
\bar\psi_{(k)}(f-f_0) + O \bigl(n^{-1} \bigr).
\]
Moreover, if
$\| \omega- \omega^0\|_1 \leq M\sqrt{k \log n}/{\sqrt{ n} } $,
\[
\bigl\| \zeta- \omega^0\bigr\|_1 \leq M \sqrt{k \log n}/{
\sqrt{n} } + \frac{2
|t|\| \tilde\psi\|_\infty}{\sqrt{n}} \leq(M+1) \frac{\sqrt{k
\log
n}}{ \sqrt{n}}
\]
and vice versa. Hence, choosing $M$ large enough (independent of $k$)
such that
\[
\Pi \bigl[ \bigl\| \omega- \omega^0\bigr\|_1 \leq(M-1)\sqrt{k\log
n/n} |Y^n,k \bigr] = 1 + o_p(1)
\]
implies that if $\sum_{j=1}^k \alpha_j = o(\sqrt{n})$, 
noting $\| \bar\psi_{(k)} \|_{L,k}=
\| \tilde\psi_{[k]} \|_{L}$,
%
\begin{equation}
\label{intermk} E^\Pi \bigl[ e^{ t\rn(\psi(f)-\hat\psi_k) } | Y^n,
A_{n,k} \bigr] = e^{ t^2\| \tilde\psi_{[k]} \|_L^2/2 } \bigl( 1 + o(1) \bigr).
\end{equation}
The last estimate is for the restricted distribution $\Pi[\cdot|
Y^n, A_{n,k}]$, but \eqref{misspecified:1} implies that the
unrestricted version also follows.
Since $\|\tilde\psi\|_L^2$ is the efficiency bound for estimating
$\psi
$ in the density model, \eqref{bias1} follows.

Now we turn to the random $k$ case. The previous proof can be
reproduced $k$ by $k$, that is, one decomposes the posterior $\Pi
[\cdot
| Y^n, B_n]$, for $B_n=\bigcup_{1\le k\le n} A_{n,k} \cap\{
f=f_{\omega,k}, k\in\cK_n\}$, into the mixture of the laws $\Pi
[\cdot
| Y^n, B_n, k]$ with weights $\Pi[k| Y^n]$.
Combining the assumption on $\cK_n$ and \eqref{misspecified:1} yields
$\Pi[B_n| Y^n]=1+o_p(1)$. Now notice that in the present context
\eqref{intermk} becomes
\begin{eqnarray*}
E^\Pi \bigl[ e^{ t\rn(\psi(f)-\hat\psi_k) } | Y^n, B_n, k
\bigr] &=& E^\Pi \bigl[ e^{ t\rn(\psi(f)-\hat\psi_k) } | Y^n,
A_{n,k}, k \bigr]
\\
& =& e^{ t^2\| \tilde\psi_{[k]} \|_L^2/2 } \bigl( 1 + o(1) \bigr),
\end{eqnarray*}
where it is important to note that the $o(1)$ is uniform in $k$. This
follows from the fact that the proof in the deterministic case holds
for any given $k$ less than $n$ and any dependence in $k$ has been made
explicit in that proof. Thus,
\begin{eqnarray*}
E^\Pi \bigl[ e^{ t\rn(\psi(f)-\hat\psi) } | Y^n, B_n
\bigr] &=& \sum_{k\in\cK_n}E^\Pi \bigl[
e^{ t\rn(\psi(f)-\hat\psi_k) } | Y^n, A_{n,k}, k \bigr] \Pi \bigl[k|
Y^n \bigr]
\\
& =& \bigl( 1 + o(1) \bigr) \sum_{k\in\cK_n}
e^{ t^2 V_k/2 + t\rn(\hat\psi
_k-\hat
\psi)} \Pi \bigl[k| Y^n \bigr].
\end{eqnarray*}
Using \eqref{bias1} together with the continuous mapping theorem for
the exponential function yields that the last display converges in
probability to $e^{t^2 V/2}$ as $n\to\infty$, which leads to the BvM theorem.\vadjust{\goodbreak}

We apply this to the four examples. First, in the case of Example~\ref
{exa-lin} with deterministic $k = K_n$, we have by definition that
$\tilde r (f,f_0) = 0$ and $\sqrt{n}( \hat\psi_{K_n} - \hat\psi) =
b_{n,K_n}+ o_p(1)$ with $b_{n,K_n} = O(\sqrt{n}K_n^{-\beta- \gamma})
=o(1)$ if $\beta+ \gamma> 1$, when $a \in\mathcal C^\gamma$. On the
other hand, if $a(x) = \1_{x\leq z}$, for all $\beta>0$,
\[
|b_{n,K_n}| \lesssim\sqrt{n} \biggl\vert\int_{\lfloor K_nz\rfloor/K_n}^z
\bigl(f_0(x) - k w^0_{\lfloor K_nz\rfloor} \bigr)\,dx \biggr\vert= O
\bigl(\sqrt{n}K_n^{-(\beta+ 1)} \bigr)=o(1).
\]

We now verify \eqref{bias1} together with \eqref{rem-rand} for Examples
\ref{exa-ent}, \ref{exa-sqr} and \ref{exa-pow}. We present the proof in
the case Example~\ref{exa-ent}, since the other two are treated
similarly. Set, in the random $k$ case
\[
\cK_n= \bigl\{ k\in \bigl[1,k_1 k_n(\beta)
\bigr], \exists f\in\cH_{k}^1, h(f, f_0) \leq
M\varepsilon_n(\beta) \bigr\},
\]
for some $k_1, M$ large enough so that $\Pi[\cK_n| Y^n]=1+o_p(1)$
from \eqref{misspecified:1}, with $\varepsilon_n(\beta) = (n/\log
n)^{-\beta/(2 \beta+ 1)}$. For $\beta>1/2$, note that
$k \veps_{n,k}^2 \leqa k \veps_n(\beta)^2 = o(1)$, uniformly over
$k\leqa k_n(\beta)$. In the deterministic case, simply set $\mathcal
K_n = \{K_n\}$.

First, observe that for $k\in\cK_n$, the elements of the set $\{f \in
\cH_k^1, h(f, f_0) \leq M \veps_n(\beta)\}$ are bounded away from $0$
and $\infty$. Indeed, since this is true for $f_0$, writing the
Hellinger distance as a sum over the various bins leads to $\sqrt{f(x)}
\geq\sqrt{c_0} - \veps_{n,k} \sqrt{k}$ which implies that $f(x)
\geq
c_0/2$ for $n$ large enough, since $k \veps_n^2 = o(1)$. Similarly, $\|
f\|_\infty\leq2 \|f_0\|_\infty$ for $n$ large. Now, by writing $\log
(f/f_0)=1+ (f-f_0)/f_0+\rho(f-f_0)$, and using that $f/f_0$ is bounded
away from $0$ and $\infty$, one easily checks that $|\tilde r(f,f_0)|$
in Example~\ref{exa-ent} is bounded from above by a multiple of $\int_0^1 (f-f_0)^2$, which itself is controlled by $h(f,f_0)^2$ for $f,f_0$
as before. Also $\rn\veps_{n,k}^2=o(1)$ when $\beta>1/2$, which implies
\eqref{rem-rand}. It is easy to adapt the above computations to the
case where $k=K_n = O(\sqrt{n}/(\log n)^2 )$.

Next, we check condition \eqref{bias1}. Since $\tilde\psi= \log f_0 -
\psi(f_0)$, under the deterministic $k$-prior with $k=K_n = \lfloor
n^{1/2}(\log n)^{-2}\rfloor$ and $\beta>1/2$,
\[
\biggl\vert\int_0^1\tilde \psi(f_0 -
f_{0[k]}) \biggr\vert= \biggl\vert\int_0^1(\tilde
\psi-\tilde\psi_{[k]}) (f_0 - f_{0[k]}) \biggr\vert
\lesssim h^2(f_0, f_{0[k]})= o(1/\sqrt{n}).
\]
In that case, the posterior distribution of $\rn(\psi(f)-\hat\psi)$ is
asymptotically Gaussian with mean $0$ and variance $\| \tilde\psi\|
_L^2$, so the BvM theorem is valid.

Under the random $k$-prior, 
recall from the reasoning above that any $f$ with $h(f, f_0) \leq
M\veps_n(\beta)$ is bounded from below and above, so the Hellinger and
$L^2$-distances considered below are comparable. For a given $k\in\cK
_n$, by definition there exists $f^*_k\in\cH_k^1$ with $h(f_0, f^*_k)
\leq M \veps_n(\beta)$, so
using \eqref{conc:rand},
%
\begin{eqnarray*}
h^2(f_0, f_{0[k]})& \lesssim&\int
_0^1 (f_0 - f_{0[k]})^2
(x)\,dx \leq\int_0^1 \bigl(f_0 -
f_{k}^* \bigr)^2 (x) \,dx \lesssim h^2
\bigl(f_0, f_k^* \bigr) \\
&\lesssim&\veps_{n}^2(
\beta).
\end{eqnarray*}
%
This implies, using the same bound as in the deterministic-$k$ case,
\[
F_0 \bigl((\tilde\psi_{[k]}- \tilde\psi)^2
\bigr) \lesssim h(f_0, f_{0[k]})^2 = O \bigl(
\veps_{n}^2(\beta) \bigr),
\]
and that $\vert F_0(\tilde\psi_{[k]}^2 ) - F_0(\tilde\psi
^2)\vert
=o(1)$, uniformly over $k \in\cK_n$. To control the empirical process
part of \eqref{bias1},
that is the second part of \eqref{csimp}, one uses, for example, Lemma~19.33 in \cite{aad98}, which provides an upper-bound for the maximum,
together with the last display.
So, for random $k$, the BvM theorem is satisfied if $\beta>1/2$.


\begin{supplement}[id=suppA]
\stitle{Supplement to ``A Bernstein--von Mises theorem for smooth functionals in
semiparametric models''}
\slink[doi]{10.1214/15-AOS1336SUPP} 
\sdatatype{.pdf}
\sfilename{aos1336\_supp.pdf}
\sdescription{In the supplementary material, we state and prove several technical
results used in the paper and provide the remaining proofs.}
\end{supplement}


%




\printaddresses
\end{document}